\documentclass[fleqn, 12pt]{amsart}

\usepackage{amsmath,amsfonts}
\usepackage[mathscr]{eucal}
\usepackage{array}
\usepackage{graphicx}
\usepackage{multirow}
\usepackage{color}
\usepackage[authoryear,sort]{natbib}

\allowdisplaybreaks[1]
\newtheorem{theorem}{Theorem}[section]

\newtheorem{proposition}[theorem]{Proposition}

\makeatletter
\@addtoreset{equation}{section}
\makeatother
\renewcommand\theequation{\arabic{section}.\arabic{equation}}
\newcommand{\RR}{{\mathbb R}}

\newcommand{\de}{\mathrm{d}}
\newcommand{\iac}{\mathrm{i}}
\renewcommand{\exp}{\operatorname{exp}}
\newcommand{\heaviside}{\operatorname{H}}
\newcommand{\sign}{\operatorname{sgn}}
\newcommand{\emath}{\mathrm{e}}
\newcommand{\ra}{\rightarrow}
\begin{document}

\title[The effect of forcing confidence]{Forcing confidence: a Process Tracing approach with a dynamical systems model}
 
\date{\today}
 
\author[A. Gheondea-Eladi]{Alexandra Gheondea-Eladi}
\address{Institutul de Cercetare a Calit\u a\c tii Vie\c tii, Academia Rom\^an\u a, Str. 13 Septembrie 1918, 
nr. 13, room 2.345, 050711 Bucure\c sti, Rom\^ania}
\email{a.gheondea.eladi@gmail.com}
\author[A. Gheondea]{Aurelian Gheondea} 
\address{Institutul de Matematic\u a ,,Simion Stoilow" al Academiei Rom\^ane, C.P.\ 
1-764, 014700 Bucure\c sti, Rom\^ania \emph{and} Department of Mathematics, Bilkent University, 06800 Bilkent, Ankara, 
Turkey} 
\email{A.Gheondea@imar.ro \textrm{and} aurelian@fen.bilkent.edu.tr } 

\begin{abstract}
We propose a continuous time dynamical system model for tracing the evolution of confidence in a small 
decision making group by consensus with the possibility that a forcing factor is exerted onto the 
confidence of the participants. We experimentally check whether a forcing factor appears, where it 
begins and ends and how it affects the evolution. We find that the equilibrium value of the 
confidence is lower in the model with the forcing factor than without it, and that the forcing factor 
can be identified and induces additional oscillations of the confidence level. This is probably one of 
the first times when a mathematical model is able to speak about visible effects on the confidence process 
tracing under 
alterations of its levels. Pragmatically we find a model that captures influence of participants' confidence level by 
observing oscillations and equilibrium and we experimentally test it with measures of individual "confidence that 
the decision is correct" throughout the group decision-making.
 \end{abstract} 

%\keywords{VH-space, VE-space, admissible space, topologically admissible space}
\maketitle 

\section*{Introduction}

	Large-scale group decision-making is frequently subject to artificial modulations of participant confidence, often occurring independently of evidence-based accuracy (Liu, Xu, \& Herrera, 2019; Zarnoth \& Sniezek, 1997; Friedkin \& Johnsen, 1990). To preserve autonomous and informed collective processes, it is essential to distinguish between normative evidence accumulation and decisions driven by confidence levels that are decoupled from information quality (Lebreton et al., 2018; Votruba \& Kwan, 2015; Patalano \& LeClair, 2011). Because large-scale social decisions are often clustered within smaller entities, such as committees, families, or reference groups (Bernardo et al., 2021), identifying and representing these artificial confidence increases is critical for ensuring the autonomy of both small-group and aggregate outcomes. Over the past two decades, established theories of decision-making have been increasingly challenged by the complex dynamics of social influence in both physical and virtual networks (Denrell, 2008; Aral \& Walker, 2012; Butts, 2016).

This paper develops and validates a dynamical systems model that differentiates between cases where an external "forcing factor" acts upon participant confidence versus cases characterized by unconstrained consensus. Our framework reveals that the presence of a forcing factor induces a shift in equilibrium values and introduces unique oscillatory behaviours in the evolution of confidence (EoC). We specifically model scenarios in which participants seek consensus while simultaneously influencing one another's confidence levels. We define confidence as the subjective probability that a chosen response is correct, though we introduce additional controls for individual and group reference points in our methodology. In this context, we distinguish between subjective certainty (the decision-maker's introspection) and objective certainty, the latter being defined as the ratio between required knowledge and the knowledge actually possessed by the decision-maker (Zamfir, 2005). While previous work has addressed these concepts (Zamfir, 2005; Gheondea-Eladi, 2016; Gheondea-Eladi \& Gheondea, 2022), this paper clarifies the distinct vocabularies used in engineering (information fusion) and psychology to describe these internal states.

Investigating the impact of a forcing factor allows us to identify decisions influenced by artificial elevations in certainty, independent of available information, by tracing the process dynamics of confidence during discussion. Observed oscillations in the EoC may stem from group influence, information exchange, or increases in self-confidence (Zarnoth \& Sniezek, 1997; Patalano \& LeClair, 2011; Viscusi, Philips, \& Kroll, 2011) and solution certainty (Votruba \& Kwan, 2015). Currently, empirical and analytical methods cannot fully isolate whether these oscillations arise from steady information reception or artificial confidence boosts; therefore, a formal mathematical model is required to highlight these specific behaviours. Building on previous findings (Gheondea-Eladi, 2016), which showed that small-group EoC oscillates toward an equilibrium in approximately two-thirds of participants, expanded by more recent work (Gheondea-Eladi \& Gheondea, 2022) identifying a model bifurcation, in this paper, we introduce a forcing factor influencing participant confidence. This bifurcation produces two distinct states, oscillatory and non-oscillatory, allowing us to mathematically represent the two temporal dynamics of confidence (see Figure 1). Our results clarify: (a) the temporal onset and termination of the forcing factor, (b) the resulting impact on equilibrium values, and (c) the manifestation of forcing through additional EoC oscillations.

The forcing factor has significant cross-disciplinary applications, reflecting phenomena such as groupthink (Janis \& Mann, 1977) and nonverbal leadership cues (Bixtera \& Luhmann, 2020; Chemers, 2004; Locke \& Anderson, 2015; Estrada \& Vargas-Estrada, 2013). Instances of over- and under-confidence can be explained by EoC oscillations where discussions conclude before reaching equilibrium (Zarnoth \& Sniezek, 1997; Gheondea-Eladi, 2016). Unlike participation-based models, in which the forcing factor is the extent to which an agent participates actively or to which other agents influenced by them propagate their influence (Zarnoth \& Sniezek, 1997; Bernardo et al., 2021), our forcing factor focuses on dynamics acting directly upon confidence regardless of knowledge. Furthermore, because the relationship between self-reported confidence and group influence is mediated by task type and aggregation method (Zarnoth \& Sniezek, 1997), we can expect that if the likelihood of a group choosing any Social Combination model (Zarnoth \& Sniezek, 1997) is normally distributed in our experiment, half of the groups had confident and accurate/inaccurate members influence the decision and the rest did not. This gives additional evidence attesting to the importance of integrating the forcing factor into EoC models of group decision.

Finally, the proposed model addresses specific constraints in mathematical modeling. While it has been claimed that confidence dynamics in consensus groups follow an ascending oscillatory path (Zamfir, 2005), empirical evidence suggests these dynamics require an initial value problem based on a third-order differential equation (Gheondea-Eladi, 2016). We resolve the difficulty of motivating such a complex model by applying the "consensus effect" principle (Hoch, 1987, as cited in Stanovich, 1998) to aggregate two simpler models: a first-order model for objective certainty and a second-order model for subjective certainty. In non-perturbed cases, this aggregation is obtained via a coupling constant (Gheondea-Eladi \& Gheondea, 2022). In the presence of a forcing factor, we utilize a coupling function derived from the objective certainty solution, treating the first-order model as the rigid component and the second-order oscillatory model as the flexible component. To maintain computational tractability, we specify the forcing term as a linear combination of Heaviside functions, treating the resulting solutions as generalized distributions. In Figure~\ref{f:samples} we represent the solutions of the model with and without the forcing factor, for two participants. In the Results section we present the empirical reasons for which the new model closer represents the data. This fit is made visible in Figure~\ref{f:samples}.

\begin{figure}
\includegraphics[width=6cm]{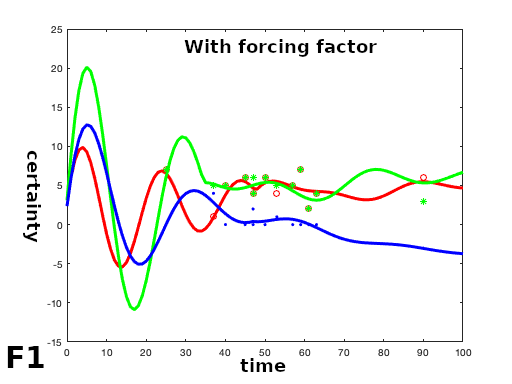}
\includegraphics[width=6cm]{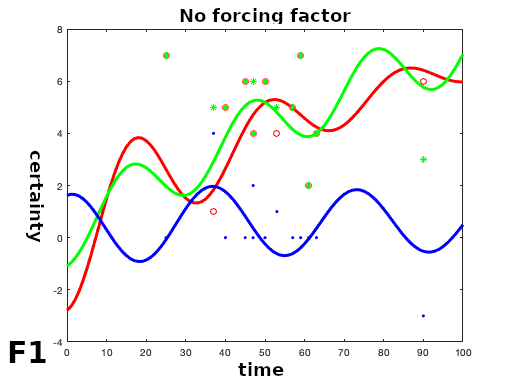}

\includegraphics[width=6cm]{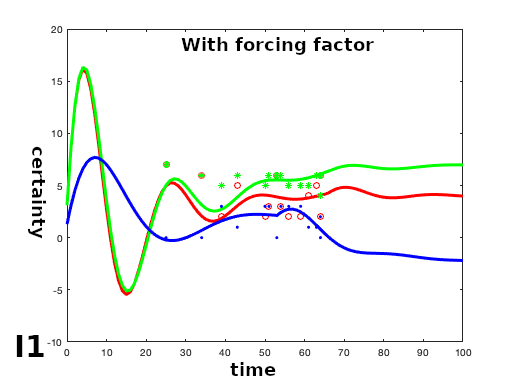}
\includegraphics[width=6cm]{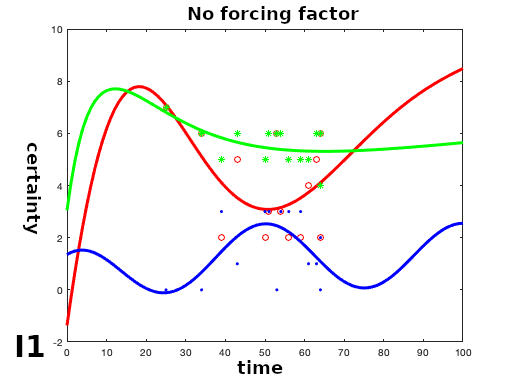}
\caption{A  comparison of the solutions on real data with the forcing factor model, on the left side, and without the forcing factor model, on the right side for two group participants identified as F1 and I1. The colors reflect three different ways of measuring the confidence, called CII, CIF and CI and described in the Methology section.}\label{f:samples}
\end{figure}

 \medskip 

%In order to ensure autonomous and informed group decision-making, it becomes important to identify group decisions 
%which are influenced in ways that are not connected to the accuracy of information. When dealing with decisions without known solution, being able to represent such influences could lead to major breakthroughs in autonomous group decision-making. This topic is particularly important for both large scale and small scale decision-making. Large scale group decision-making is sometimes influenced by artificial increases in participants' confidence and self-confidence irrespective of the accuracy 
%of information. At the same time, large scale group decision-making in social settings is many times done by clustering decisions in small groups, like committees, comissions, families, 
%personal reference groups or communities.  In this context,

\subsection*{Previous and Related Work.}
 Previous studies in group decision-making provide important contributions on artificial increases of members' confidence during group decisions, aggregation and fusion of preferences and information, and the use of dynamical systems for social psychology applications. In this section we provide a short overview of each of these.

{\bf{The Psychological Foundations of Group Consensus}}
Consensus-based group decisions represent a complex interplay between the exchange of relevant information and the modulation of participant confidence, which are not always correlated.  When consensus is reached independently of evidentiary quality, the resulting decisions may be considered "artificial". Such phenomena include  groupthink (Janis, 1977), over- and under-confidence (Dunning et al., 1990; Oskamp, 1982; Ronay et al., 2017) and deferral of decisions (Estrada \& Vargas-Estrada, 2013; Locke \& Anderson, 2015). Research has identified several proxies that frequently replace or distort information evaluation, such as solution frequency, group homogeneity (Krueger, 1998; Mackie, 1987), physiological linkage (Thorson et al., 2019). Social dynamics further modulate this process through peer pressure (Hoffman et al., 2001), entitativity (Campbell, 1958), and group polarization (Fraser, 1971; Zhu, 2014). Moreover, the influence of leadership, authority, and status (Johnson \& Ewens, 1971; Bixtera \& Luhmann, 2020; Chemers, 2004; Kwaadsteniet \& Dijk, 2010; Liebe \& Tutic, 2010) and targeted persuasion (Evans \& Clark, 2012) can drive participants toward consensus at confidence levels significantly lower than their initial internal states (Estrada \& Vargas-Estrada, 2013).

{\bf{Factors influencing subjective certainty (confidence)} }
Even when an objective solution is unknown, consensus can be maintained through high levels of subjective certainty or trust in the decision-making process.  Key factors influencing this internal confidence include advice-taking behavior (Hutter \& Fiedler, 2019; Larson et al., 2019) and the temporal sequencing of doubt and confidence expressions (Wichman et al., 2009). The framing of messages in confidence terms, particularly after priming, exerts a significant impact on final judgments (Tormala et al., 2008). These responses may be overt or hidden depending on a participant’s perception of their standing as a majority or minority member within the group (Noelle-Neumann, 1993; Thomas et al., 2015). Estrada and Vargas-Estrada (2013) present a model that captures the peer pressure as a function of the socio-cultural distance between individuals in a social group and show that the peer pressure level determines how fast a social group reaches consensus and that the levels of peer pressure determine the leaders who can achieve full control of their social groups. From a different perspective, Hutter \& Fiedler (2019) investigated the extent to which framing the very same information as either advice or anchor exerts a differential influence on judgments, in both humans and computers, and they found different levels of influence between advice and anchor. Based on bounded rationality and groupthink,  Liu et al. (2019) proposed a model with fuzzy preference relations with selfconfidence.

{\bf{Mathematical Frameworks for Information Fusion}}
From a formal perspective, group decision-making requires the fusion of heterogeneous preference structures to reach a unique agreement (Salerno, 2002; Bostrom et al., 2007).  The literature has addressed this through linguistic approaches (Perez, 2011), social network analysis (Estrada \& Vargas-Estrada, 2013), and fuzzy rank-level fusion (Sing et al., 2019). Traditional paradigms in multi-attribute decision making (MADM) focus on utility preferences and preference relations (Dong \& Xu, 2016) as well as fractional nonlinear dynamics (Abdulghafor \& Turaev, 2018), fuzzy relations with self-confidence parameters (Liu et al., 2019), and feedback mechanisms designed for minimum adjustment in large-scale systems (Zhang et al., 2020b). While common methodologies aggregate group variables (Okada \& Lee, 2016; Stengel, 2013; Schultze et al., 2012; Kraft et al., 2002; Galam \& Zucker, 2000) or utilize standard information fusion (Shuping et al., 2018; Jiao \& Li, 2021; Zhang et al., 2020a), our model focuses on the aggregation of certainty. We treat the final group outcome as the terminal point of an individual’s internal evolution of confidence (EoC), thereby modeling the temporal trajectory of confidence during social interaction rather than focusing solely on individual-group discrepancies.

{\bf{Dynamical Systems in Social Psychology}}
To trace this process, we utilize a Dynamical Systems (DS) model to represent the continuous evolution toward a decision threshold. DS frameworks, whether formulated as discrete-time recursive sequences or continuous-time ordinary differential equations (Lynch, 2004; Pham et al., 2018; Sadovnichiy \& Zgurovsky, 1986), are robust tools across both engineering and economics (Medio \& Lines, 2001). As surveyed by Hassani et al. (2022), classical consensus models such as the {\em{Ising}}, voter, and DeGroot models provide a foundation for understanding opinion dynamics.
While DS applications in social psychology emerged in the ‘80s and 90’s in fields like  perception and development (Vallacher \& Nowak, 1997), they have expanded into personality (Vallacher et al., 2002), women's sexual desire (Diamond, 1986, 2007), and clinical/educational contexts (Hosenfeld et al., 2015; Koopmans \& Stamovlasis, 2016). Social-level DS models have elucidated opinion dynamics (Deffuant et al., 2000; Friedkin \& Johnsen, 1990), network structures (Estrada \& Vargas-Estrada, 2013), and voting behavior (Jiao \& Li, 2021). Further DS approaches can be 
linked to interval valued fuzzy preference relations models, e.g. through DeGroot social network models or other (Shuping et al., 2018; Zhang et al., 2020b). An analysis of the dynamics of consensus building in group decision making by searching the consensus path with minimum adjustments lead researchers to use the mathematical theory of optimisation which has also been employed for a very long time as an extraordinary rich source with very powerful results (Dong and Xu, 2016). More recent mathematics have been used in the analysis of consensus, such as models with nonlinear dynamical systems, the speed of convergence to consensus, and fractional calculus (Abdulghafor and Turaev, 2018), feedback mechanism with minimum adjustment and cost (Zhang et al., 2020b), active opinion dynamics combined 
with Bayesian analysis (Jiao and Li, 2021) and quantum probability and Choquet integral (Xiao et al., 2024).

\section{The Mathematical Model}\label{s:md}

The DS model that we employ in this research relies on initial value problems associated to ordinary differential equations and, among other things, it has the advantage that allows simulations that
emphasise the bifurcation and an 
analysis of the stability to numerical perturbations. Another important advantage of the DS 
model is that the general model can be obtained through a subtle aggregation procedure of the objective certainty, which corresponds to a differential equation of order one, and the subjective 
certainty, which corresponds to a differential equation of order two. The resulting general model thus 
corresponds to a differential equation of order three that encompasses both the evolution of confidence and
the possible oscillatory behaviour. In the absence of the modelling process employed here, only a numerical solution of the third-order differential equation would 
have been possible and this would be less useful in understanding the EoC.

 Our approach is based on a DS model and there 
are a number of reasons to do so. Firstly, this model has the capacity to capture the oscillations of the EoC, 
especially in the case
of subjective certainty. The eight parameters that we use in the model for subjective confidence
have clear interpretations right from
the beginning and all the mathematical and numerical operations preserve these parameters in an explicit 
fashion. It is this trait that enables us to connect the oscillatory or non-oscillatory behaviour to the balance between sensitivity to the level of the confidence and 
the sensitivity to the rate of change of confidence: for a definition see (Gheondea-Eladi and Gheondea, 2022).  Secondly, the DS model has the potential of encoding a further analysis on 
overconfidence (Dunning et al., 1990; Plous, 1995; Giardini et al, 2008; Oskamp, 1982) that might show up in the EoC in decision-making groups by consensus. 
In this paper, this is done explicitly, by introducing more parameters which cover overconfidence by the forcing factor. 
And, thirdly, the DS model can be 
matched with existing models for decision-making in 
social networks which use systems of differential equations (Estrada and Vargas-Estrada, 2013). In this way, our model based on a single differential equation has the potential 
to be adapted in future studies of the EoC in social networks. 

\subsection*{Rationale for the components of the dynamical system model}
The rationale for the components of the dynamical system model are connected to the process traits of the EoC we intended to capture, the modelling techniques we employed (such as superposition), the type of individual interaction between individuals in the group and the theoretical background that is used as the basis of our model (subjective and objective confidence). In this section we will explain each of these. 
\begin{enumerate}
\item{{\bf{Expected behaviour of the certainty/confidence function.}}} From the onset, the model intended to capture oscillation and evolution to equilibrium (two mathematically qualitative aspects that have direct psychological interpretation), as predicted by (Zamfir, 2005). 
\item{{\bf{Modelling technique.}}} Knowing that oscillation and evolution to equillibrium occur in solutions to third order differential equations (ODE), we intended to construct such a model. However, for third order differential equation models it is usually very difficult to calculate the analytical solutions, while the interpretation of parameters may become concealed. 
\item{{\bf{Modelling individuals' interactions.}}} In order to solve the third order ODE problem, we modelled group interaction as a market-type interaction between providers and receivers of information.
\item{{\bf{Modelling certainty.}}} We used the individual interaction model to build both the first order and the second order model (subjective and objective certainty) that can be superposed to obtain the third order model. The rationale for which we can use this technique is based on Hoch’s principle (explained on pg. 2). The concepts of subjective and objective certainty, proposed by Zamfir (2005) have been detailed in Gheondea-Eladi (2016). 
\end{enumerate}
To support the reading of the following mathematical modelling subsection, we will summarize here the parameters that require psychological interpretation and which will be employed for the mathematical model:
\begin{itemize}
\item $\alpha$, which represents the sensitivity to the rate of change of certainty,

\item $\beta$, which represents the sensitivity to the absolute value of the certainty,

\item $\gamma$, which represents the level of certainty at equilibrium,

\item $n$ and $m$, the endpoints of the interval of time when the forcing factor appears,

\item $p$, the amplitude of the forcing factor.

\end{itemize}

In sections 1.1, 1.2, and 1.3 these psychological interpretations of the parameters in the model are carefully described in the process of mathematical modelling and the hypotheses that lead to the mathematical model are explicitly presented.

Furthermore, in order to connect the theorems and proofs to the empirical claims we need to follow the four stages of the mathematical modelling process: Formulating the Real Life Problem (presented in the Introduction, paragraph 1), Formulating the Mathematical Problem (summarised in the introduction, pg. 2, par. 2 and detailed in section 1), Finding the Mathematical Solution (Section 1.3), and Finding the Real Life Solution (Section 2.3 and 2.4). Furthermore, we purposefully structured the data analysis in the following manner: the subsection title makes reference to the Real World Problem (e.g. \lq When does the forcing factor begin and end?\rq ) and the first or second sentence in the section indicates which model parameter is connected to the behaviour we wanted to observe. For example, in section 2.4.1 labelled: Is there an observable forcing factor? we state: \lq we provide a histogram of the values of $p$, which represents the amplitude of the oscillation of the forcing factor\rq . The value $p$ is one of the parameters of the general model given in equations (1.29) and (1.30). We do the same for the other sections.

\subsection*{To fit or not to fit the model}
For the readers accustomed to seeing model fit analyses, there is an important note to make. Model fitness, as statistical analysis of the distances between the model and the empirical data would not add anything to our research question’s answers, since we do not aim to show that this is the best model of all, but that this model captures the desired behaviours. Our research question is mathematically qualitative. We consider this to be an even more important question than the quantitative one in light of the existing models (reviewed in the section on Previous and Related Work) which do not aim to capture these behaviours that the theory presents to be important. The value and the difficulty of the model we propose is given by two aspects. 
\begin{enumerate}
\item We have obtained an analytical model and not a statistical one and therefore we understand exactly the interpretations of the limits and capabilities of this model as well as the meaning of each parameter. 
\item We have obtained a model that captures exactly the aspects that were predicted theoretically and nothing extra that has no theoretical or mathematical interpretation. This is a major step that allows for further testing. In such cases the adequate measure of fit is with respect to the goals of the model, i.e. whether it captures the aspects predicted or not.
\end{enumerate}

\subsection{A Mathematical Model for Objective Certainty.}\label{ss:mmoc}
We consider a small group of people in a decision making process by consensus,
on a single issue (problem), in which the communication flow is a 
continuous variable, mathematically identified with the real positive
semi-axis, while the certainty level is a function of
communication. In this state space approach, all the participants and the subgroups are \emph{providers}
(suppliers) of information as well as \emph{receptors} 
(buyers) of information. The certainty with respect to the decision-making problem is considered 
a measurable quantity that reflects a compound concept. 
Briefly, we consider the following functions: $R(\theta)$ the \emph{receptor function}, 
and $P(\theta)$  the \emph{provider function}, 
both in the variable $\theta$ which denotes the \emph{certainty}. 
Both functions $R$ and $P$ are real valued one real variable functions in the state space describing the 
dynamical system of the decision-making process. 

For the beginning we consider the \emph{objective certainty} for which
the first assumptions on $R$ and $P$ are the following:

\begin{enumerate}
\item[(Ao)] Both functions $R$ and $P$ are linear with respect to the
  variable 
  $\theta$.\medskip
\item[(Bo)] The function $R$ is decreasing while the function $P$ is
  increasing with respect to $\theta$.
\end{enumerate}

The linearity of $R$ and $P$ is just a first
order approximation in order to keep the model as simple as possible. The second assumption
concerns the following reasoning. First, \emph{the higher the certainty
the lower is the receptor's interest in getting more information and the higher
is the provider's interest in offering more information}. Second, \emph{the
lower the certainty, the higher is the receptor's interest in getting
more information and the lower is the provider's interest in offering more
information}. 

We now introduce the following assumption which makes the difference when compared to the
previous article:

\begin{enumerate} 
\item[(Co)] In the communication process, the provider, deliberately or not, supplies,
starting at some time $m$ and stopping at some time $n>m$, a surplus of certainty of level $p$. 
\end{enumerate}
Mathematically this is modelled by a step function $f$
\begin{equation}\label{e:ff}
f(t)=\begin{cases} p, & m< t <n,\\ 
\frac{p}{2}, & t=m,n, \\ 0, & \mbox{ otherwise.}\end{cases}
\end{equation}
It is preferable to express the function $f$ in terms of the Heaviside function $H$
\begin{equation}\label{e:has}
H(t)=\begin{cases} 0, & t<0, \\ \frac{1}{2}, & t=0, \\
1,& t> 0,\end{cases}
\end{equation}
hence
\begin{equation}\label{e:feh}
f(t)=p\bigl(H(t-m)-H(t-n)\bigr),\quad t\in\RR.
\end{equation}

From these assumptions it follows that the formal
representations of the functions $R$ and $P$ are:
\begin{equation}\label{receptorprovider} R(\theta)=a -b\theta,\quad
  P(\theta)=c+d(\theta+f),\end{equation}
where all $a,b,c,d$ are positive real numbers: $b$ and $d$ are positive
because of the decreasing/increasing behavior of $R$ and $P$, while $a$ and
$c$ are positive because it is important to grant some nontrivial certainty at the
beginning (they can be taken nonnegative as well, but information with null
certainty does not seem to be a realistic premise). This restriction is also
based on some experimental observations (see Section~\ref{s:da}) and on
the problems imposed by the existence of \lq \lq irreducible uncertainty \rq \rq (Zamfir,
2005: 64).   

Then, as in (Gheondea-Eladi, 2016)
\begin{enumerate}
\item[(Do)] As the communication flows, the objective certainty $\theta$
considered as a function of communication $t$ changes  proportionally with the 
\emph{excess of certainty of the receptor over the provider}, that is, with
$R-P$. 
\end{enumerate}
In a mathematical formulation, this assumption can be written as the
following equation for the adjustment of the \emph{objective certainty}: 
\begin{equation}\label{adj}
  \theta(t+h)=\theta(t)+hk[R(\theta(t))-P(\theta(t))],\end{equation} 
where, 
$h>0$ is the length of the interval of communication, that is, the
interval $(t,t+h)$, and $k>0$ is a positive number that describes the
\emph{rate of change of the objective certainty as the response to the excess
  receptor/provider}. Since $k$ can be absorbed into other coefficients,
without loss of generality, from now on, $k$ is assumed to be $k=1$.

We can now proceed to determine a first order dynamical system approach that is suitable for \emph{objective
  certainty}, which is an idealised representation of the certainty.
Let us begin with the equation for the adjustment of the objective certainty
\eqref{adj}, taking into account that $h>0$ and using \eqref{receptorprovider}, hence
\begin{equation}\label{adjquotient}
  \frac{\theta(t+h)-\theta(t)}{h}=\bigl(R(\theta(t))
-P(\theta(t))\bigr)=a-c-(b+d)\theta(t)-df(t).\end{equation}

In order to continue the description of this continuous dynamical system we
have to make our fourth assumption.

\begin{enumerate}
\item[(Eo)] The objective certainty function $\theta$ is differentiable everywhere, but a finite number of points,
with respect to the variable $t$ of communication.
\end{enumerate}

Based on assumption (Eo) and taking into account the functional
representations of $R$ and $P$ as in \eqref{receptorprovider}, we can pass to
the limit as $h\rightarrow 0$ in \eqref{adjquotient} and get that for all $t$ but a finite number of them,
\begin{equation}\label{e:diff} \frac{\de\theta}{\de t}+(b+d)\theta(t)
=(a-c)-dp\bigl(H(t-m)-H(t-n)\bigr),\end{equation} where, we always have to keep in
mind that $\theta$ is a function of $t$, and hence that the ordinary
differential equation \eqref{e:diff} is valid for all but a finite number of $t\geq 0$. In more rigorous
mathematics, the ODE \eqref{e:diff} has to be understood in the distribution (generalised functions) sense.
In
addition, this is made into an Initial Value Problem (briefly, IVP), by specifying the value of $\theta$ at $0$, 
that is $\theta(0)=\theta_0$, a given constant.

The solution to the IVP \eqref{e:diff}, under the assumptions $0<m<n$ and $a,b,c,d>0$, 
as provided by {\tt dsolve} in MATLAB, is
\begin{align*}
\theta(t) & = \exp(-t (b + d)) \biggl((a \exp(b t + d t))/(b + d) - (c \exp(b t + d t))/(b + d) \\
& + (d p \heaviside(t - m) \exp(b m + d m))/(b + d) - (d p \heaviside(t - n) \exp(b n + d n))/(b + d) \\
& - (d p \heaviside(t - m) \exp(b t + d t))/(b + d) + (d p \heaviside(t - n) \exp(b t + d t))/(b + d)\biggr) \\
& + (\exp(-t (b + d)) (c - a + b \theta_0 + d \theta_0))/(b + d)
 \end{align*}
equivalently,
 \begin{align}
\theta(t) & = \frac{a-c}{b+d}- \frac{d p}{b+d} \bigl(\heaviside(t - m) - \heaviside(t-n)\bigr) \label{e:thetaunuar}\\
& \ \ \ -\emath^{-t(b+d)} \biggl[\bigl(\frac{a-c}{b+d}-\theta_0\bigr) - \frac{dp}{b+d} \biggl(\heaviside(t - m) \emath^{(b + d) m} 
-  \heaviside(t - n) \emath^{(b  + d) n}\biggr)\biggr]. \nonumber
 \end{align}
  
\begin{figure}
\includegraphics[width=12cm]{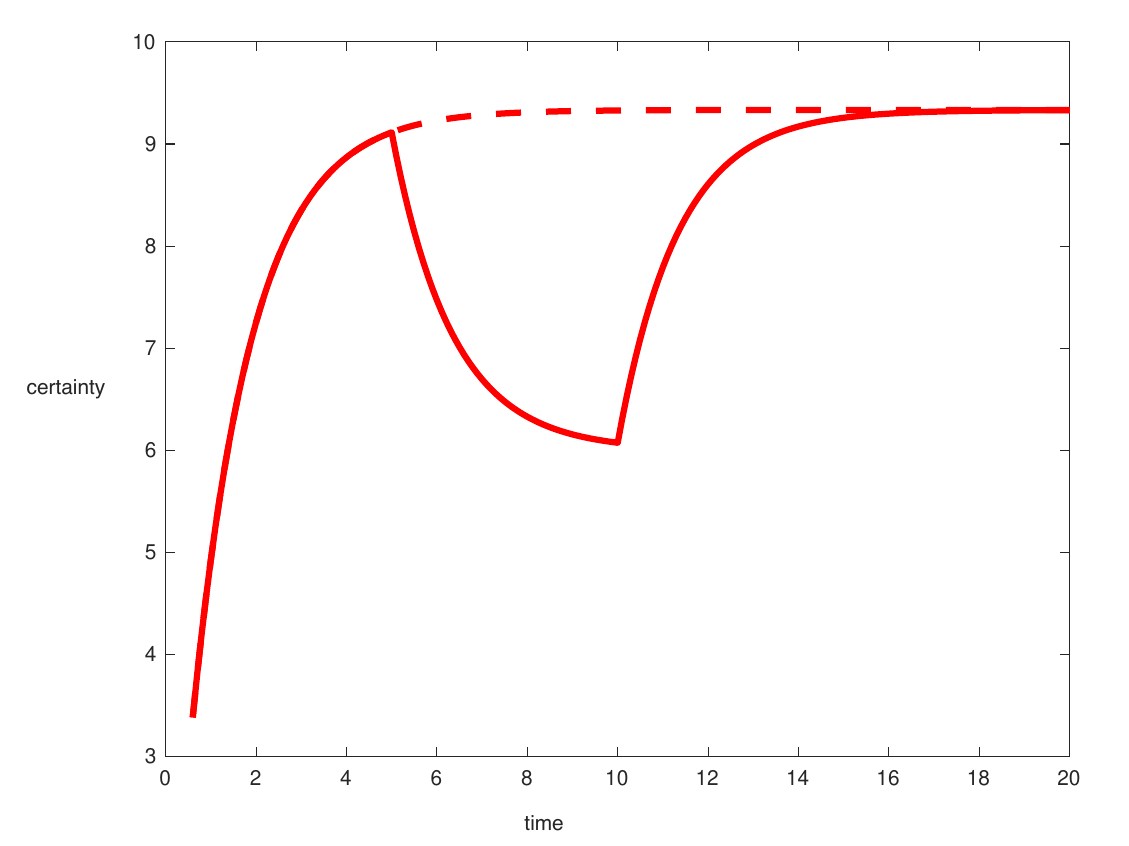}
\caption{A simulation of the evolution of the objective certainty for $a=7$, $b=0.5$, $c=0$, $d=0.25$, $m= 5$, 
$n=10$, $p=10$,  and $\theta_0= 0$ (the graph in continuous line), compared with the unperturbed objective  
certainty with $a=7$, $b=0.5$, $c=0$, $d=0.25$, $\theta_0= 0$ (the graph in dashed line).}\label{f:simord1}
\end{figure}
  
 As in the unperturbed case, see (Gheondea-Eladi, 2016), we consider the equilibrium constant
 \begin{equation}\label{e:equc}
 \overline\theta=\frac{a-c}{b+d}.
 \end{equation}
 In \eqref{e:thetaunuar} we point out on the first row the equilibrium function,
 \begin{equation}\label{e:equf}
 t\mapsto \frac{a-c}{b+d}- \frac{d p}{b+d} \bigl(\heaviside(t - m) - \heaviside(t-n)\bigr),
 \end{equation} and, on the second row we point out the exponential function with
negative exponent that ensures the evolution to equilibrium, that has as a factor the evolution function 
\begin{equation}\label{e:eva}
t\mapsto\bigl(\frac{a-c}{b+d}-\theta_0\bigr) - \frac{dp}{b+d} \biggl(\heaviside(t - m) \emath^{(b  + d) m} 
-  \heaviside(t - n) \emath^{(b  + d) n}\biggr).
\end{equation}

Comparing to the unperturbed evolution of the certainty, the novelty is that two of the constants, the 
equilibrium constant $\overline\theta$ and the evolution constant $\overline\theta-\theta_0$, 
 become functions in the forcing perturbed case. The 
equilibrium function is constant except on the interval $[m,n]$, where it has a gap. The meaning of the gap
is that, on the time interval when the forcing perturbation occurs, the equilibrium level is smaller and, 
consequently, for $n$ large, e.g.\ larger than the time when the group activity ends, it becomes the new 
equilibrium constant. From the point of view of social psychology, this means that when a forcing factor occurs participants in the group reach consensus at a certainty level which is lower than the level they would have reached if there was no forcing factor. As it has been presented in Introduction there are cases in which one could expect this to happen, such as delegation of responsibility and so on.
The evolution function \eqref{e:eva} is a step function with three levels, meaning that there is a small 
oscillation
with nodes at $m$ and $n$, see Figure~\ref{f:simord1}, where the perturbation occurs in the time
interval $[5,10]$.

\subsection{A Mathematical Model for Subjective Certainty.}\label{ss:mmsc}
A realistic assumption in the case of the \emph{subjective certainty} 
is that there is a delay in the reactions of the receptors and the providers
of information in the decision-making group, mathematically expressible by the
fact that both the receptor and the provider functions  
vary with $\theta$ and $\theta^\prime$ as well. Therefore, 
we have the following assumptions: 
\begin{enumerate}
\item[(As)] Both functions $R$ and $P$ are linear with respect to the
variable $\theta$ and its rate of change $\theta^\prime$.\medskip
\item[(Bs)] The function $R$ is separately
decreasing, while the function $P$ is
separately increasing, with respect to both $\theta$ and $\theta^\prime$.
\end{enumerate}
These assumptions are justified similarly as the assumptions (Ao) and (Bo) from the preceding subsection. As in (Gheondea-Eladi, 2016), the dependence on $\theta^\prime$ is justified by the need to take into account group members' rate of change of the certainty and not only its absolute value. More than this, the sensitivity to the rate of change of the certainty has been shown to split participants' evolution of certainty in two states --- oscillating and non-oscillating --- which have motivated the inquiry into the oscillatory behaviour resulting from the presence of a forcing factor (Gheondea-Eladi and Gheondea, 2022).

We now introduce the following assumption:
\begin{enumerate} 
\item[(Cs)] In the communication process, the provider deliberately supplies,
starting at some time $m$ and stopping at some time $n>m$, a surplus of certainty of level $p$. 
\end{enumerate}
Mathematically this is modelled by a step function $f$ and the Heaviside function as in 
\eqref{e:ff}--\eqref{e:feh}.
Based on these assumptions, we have the following formal
representations of the functions $R$ and $P$:
\begin{equation}\label{receptorprovider2} R(\theta,\theta^\prime)=a
  -b\theta-g\theta^\prime,\quad 
  P(\theta,\theta^\prime)=c+d(\theta+f)+e(\theta^\prime+f^\prime),\end{equation}
where all $a,b,c,d,e,g$ are positive real numbers.

The main assumption, in this case, is the following:
\begin{enumerate}
\item[(Ds)] As the communication flows, the \emph{rate of change of
certainty}
considered as a function of communication $t$, changes  
proportionally with the 
\emph{excess of the receptor over the provider}, that is, with $R-P$.
\end{enumerate}
In a mathematical formulation, this assumption can be written as the
following equation for the adjustment of the subjective certainty:
\begin{equation}\label{adj2}
  \theta(t+h)=\theta(t)+h\theta^\prime(t) +
 \frac{1}{2}h^2[R(\theta,\theta^\prime)-P(\theta,\theta^\prime)],\end{equation} 
where, 
$h>0$ denotes the length of a very small interval of communication, and the
factor $1/2$ is inserted, in view of the Taylor formula, for consistency of
the constants.

On the other hand, since these assumptions refer to second order approximation
of $\theta$, we need some assumptions on the smoothness of $\theta$. In this case,
we have the following:
\begin{enumerate}
\item[(Es)] The certainty function $\theta$ is twice 
differentiable everywhere, but a finite number of points, with respect to $t$.
\end{enumerate}
Mathematically, assumption (Ds) is translated by the second order
Taylor approximation of $\theta$
\begin{equation}\label{thetataylor}
  \theta(t+h)=\theta(t)+h\theta^\prime(t)+h^2\frac{\theta^{\prime\prime}(t)}{2}+\mathrm{o}(h^3), 
\end{equation} for $h$ sufficiently small, where the notation $\mathrm{o}(h^3)$ refers to other terms of
order $3$ or higher as functions of $h$, that will be ignored. Thus, from
\eqref{receptorprovider2}, \eqref{adj2}, and \eqref{thetataylor},
\begin{align} \theta(t)+h\theta^\prime(t)+ & h^2\frac{\theta^{\prime\prime}(t)}{2}+\mathrm{o}(h^3)
  = \theta(t) + h\theta^\prime(t) \nonumber \\ & \phantom{=}+ \frac{1}{2}h^2
  \biggl( (a-c)-(b+d)\theta(t)-df(t)-(e+g)\theta^\prime(t)-ef^\prime(t)\biggr)+\mathrm{o}(h^3).
\label{doua}
\end{align} 
By the uniqueness of the Taylor representations in \eqref{doua} the coefficients corresponding to $h^n$, for $n=0,1,2,\ldots$, can be identified, accordingly. For
$n=0$  and $n=1$ nothing is obtained. 
Identification of coefficients of $h^2$ in
\eqref{doua} gives us
the second order linear differential equation with constant coefficients
\begin{equation}\label{e:diffdoi} \frac{\mathrm{d}^2
    \theta}{\mathrm{d}t^2} +(e+g)\frac{\mathrm{d}\theta}{\mathrm{d}t}
    +(b+d)\theta =(a-c)+df(t)+ef^\prime(t),   
\end{equation} in the general non-homogeneous form, equivalently, in terms of the representation of the perturbation function $f$ as in \eqref{e:has},
\begin{equation}\label{e:diffdoieq}
\frac{\mathrm{d}^2
    \theta}{\mathrm{d}t^2} +(e+g)\frac{\mathrm{d}\theta}{\mathrm{d}t}
    +(b+d)\theta =(a-c)+dp\bigl(H(t-m)-H(t-n)\bigr)+ep\bigl(\delta(t-m)-\delta(t-n)\bigr),
\end{equation}
where $\delta$ denotes the Dirac "function". With notation,
that simplify the formulae and turn out to be useful since they have 
relevant interpretations from the social psychology point of view, namely,
\begin{equation}
\alpha=\frac{e+g}{2},\quad \beta=b+d,\quad \gamma=a-c,
\end{equation}
the ODE \eqref{e:diffdoieq} becomes
\begin{equation}\label{e:diffdoieqref}
\frac{\mathrm{d}^2
    \theta}{\mathrm{d}t^2} +2\alpha\frac{\mathrm{d}\theta}{\mathrm{d}t}
    +\beta\theta =\gamma+dp\bigl(H(t-m)-H(t-n)\bigr)+ep\bigl(\delta(t-m)-\delta(t-n)\bigr),
\end{equation}
We emphasise that the ODE \eqref{e:diffdoieqref} has to be understood in the distribution (generalised 
functions) sense. This 2nd order ODE is usually accompanied by the initial conditions
\begin{equation}\label{e:initial} \theta(0)=\theta_0,\quad \theta^\prime(0)=\theta^\prime_0,
\end{equation} and \eqref{e:diffdoi} together with \eqref{e:initial} make the
\emph{initial value problem}, briefly IVP.

Firstly, let us observe that, although
the ODE \eqref{e:diffdoieqref} 
is considered in the distribution sense, what we can use in concrete situations is a 
solution of function type. Secondly, from the point of view of
numerical calculations and approximation of the solution by the least square method what we have to see is 
whether the solution is robust or not. More precisely,
an analytical expression of the solution that is not robust means that small variations of the
discriminant about $0$ may cause huge errors in calculations and approximations. 
We solve these problems in the following theorems.

\begin{theorem}\label{t:2ode}
The IVP  \eqref{e:diffdoieqref} and \eqref{e:initial}
has unique solution defined as follows: if $\alpha^2-\beta\neq 0$ then
\begin{align}\label{e:2odesola}
\theta(t) & = \frac{\gamma}{\beta} - \frac{dp}{\beta}\bigl(H(t-m)-H(t-n)\bigr)\\
& + \frac{\emath^{-\alpha t}}{2\sqrt{\alpha^2-\beta}}
\biggl[ \bigl((\theta_0-\frac{\gamma}{\beta})(\alpha+\sqrt{\alpha^2-\beta})+\theta_0^\prime\bigr)  
\emath^{t\sqrt{\alpha^2-\beta}} \nonumber 
+ \bigl((\theta_0-\frac{\gamma}{\beta})(-\alpha+\sqrt{\alpha^2-\beta})-\theta_0^\prime\bigr)
\emath^{-t\sqrt{\alpha^2-\beta}}\biggr]\nonumber \\
& + \frac{ep\emath^{-\alpha t}}{2\sqrt{\alpha^2\!-\!\beta}}\biggl[H(t\!-\!m)\emath^{m\alpha}
\bigl(\emath^{-(m-t)\sqrt{\alpha^2\!-\!\beta}}\!-\!\emath^{(m-t)\sqrt{\alpha^2\!-\!\beta}}\bigr)
 - H(t\!-\!n)\emath^{n\alpha}
\bigl(\emath^{-(n-t)\sqrt{\alpha^2\!-\!\beta}}\!-\!\emath^{(n-t)\sqrt{\alpha^2\!-\!\beta}}\bigr)\biggr]\nonumber
\end{align}
and, if $\alpha^2-\beta= 0$ then
\begin{align}\label{e:2odesoldisnula} \theta(t) & = 
\frac{\gamma}{\alpha^2}-\frac{dp}{\alpha^2}\bigl[H(t-m)-H(t-n)\bigr] \\
& \phantom{\frac{\gamma}{\alpha^2}+}
+\emath^{-\alpha t}\biggl[ 
(\theta_0-\frac{\gamma}{\alpha^2})+\bigl(\theta_0^\prime-(\frac{\gamma}{\alpha^2}-\theta_0)\alpha\bigr)t %\nonumber \\
% & \phantom{\frac{\gamma}{\alpha^2}+} 
-ep(m-t)H(m-t)\emath^{m\alpha}+ep(n-t)H(n-t)\emath^{n\alpha}
\biggl] .\nonumber
\end{align}
In addition, if $\alpha^2-\beta<0$ then, letting $\omega=\sqrt{\beta-\alpha^2}$, the solution $\theta$ defined
as in \eqref{e:2odesola} has the following representation
\begin{align}
\theta(t) & = \frac{\gamma}{\beta}-\frac{dp}{\beta}\bigl(H(t-m)-H(t-n)\bigr) \label{e:2odesolcoma}\\
& \ \ \ +\emath^{-\alpha t} \times \nonumber \\
& \ \ \ \ \times\! \biggl[\! \biggl( (\theta_0 -\frac{\gamma}{\beta})\!-\!\frac{ep}{\omega}\bigl[ \emath^{m\omega}
\sin(m\omega) H(t-m)-\emath^{n\omega}\sin(n\omega)H(t-n)\bigr]\!\biggr)\!\cos(\omega t) \nonumber \\
& \ \ \ \ \ \ \ \ +\biggl(\frac{\theta_0^\prime +\alpha(\theta_0 -\frac{\gamma}{\beta})}{\omega}  +\frac{ep}{\omega}\bigl[ \emath^{m\omega}\cos(m\omega)H(t-m)-\emath^{n\omega}
\cos(n\omega)H(t-n)\bigr]\biggr)\sin(\omega t)\biggr].\nonumber
\end{align}
\end{theorem}

We defer the rather long and technical proof of Theorem~\ref{t:2ode} to the Appendix.
The fact that the solution to the IVP \eqref{e:diffdoieqref} and \eqref{e:initial} has two different analytical representations depending
whether $\alpha^2-\beta$ is null or not null, may be the source of troubles from the numerical purposes. The
next theorem says that this is not the case, more precisely, for numerical purposes it is sufficient to use
the analytical expression from \eqref{e:2odesola}.

\begin{theorem}\label{t:robust} With notation as in Theorem~\ref{t:2ode}, the solution $\theta$ is uniformly
continuous on compact subsets with respect to the parameters $b,d,e,g,m,n>0$, $m<n$,
$a,c,\theta_0,\theta_0^\prime\in\mathbb{R}$ and $t\in\RR_+\setminus\{m,n\}$.
\end{theorem}

\begin{proof}
 We first consider $\alpha^2-\beta\neq 0$ and then, by using the power series expansions of
$\emath^{t\sqrt{\alpha^2-\beta}}$, $\emath^{-t\sqrt{\alpha^2-\beta}}$, $\emath^{(m-t)\sqrt{\alpha^2-\beta}}$,
$\emath^{-(m-t)\sqrt{\alpha^2-\beta}}$, $\emath^{(n-t)\sqrt{\alpha^2-\beta}}$, and 
$\emath^{-(n-t)\sqrt{\alpha^2-\beta}}$, it follows that the analytical form of $y$ as in \eqref{e:2odesola} becomes
\begin{align}\label{e:2odesolseries}
\theta(t) & = \frac{\gamma}{\beta} - \frac{dp}{\beta}\bigl(H(t-m)-H(t-n)\bigr)
+\emath^{-\alpha t}\biggl[\theta_0-\frac{\gamma}{\beta} 
+\sum_{k=1}^\infty \frac{a_k}{k!} (\alpha^2-\beta)^{\frac{k-1}{2}} t^k\biggr] \\
& \ \ + ep\emath^{-\alpha t}\biggl[H(t\!-\!n)\emath^{n\alpha} 
\sum_{l=0}^\infty \frac{(\alpha^2\!-\!\beta)^{l}(n\!-\!t)^{2l+1}}{(2l+1)!}
\!-\!H(t\!-\!m)\emath^{m\alpha} 
\sum_{l=0}^\infty \frac{(\alpha^2\!-\!\beta)^{l}(m\!-\!t)^{2l+1}}{(2l+1)!}\biggr],
\nonumber
\end{align}
where, for all integer numbers $k\geq 0$,
\begin{equation}
a_k=\begin{cases} (\theta_0-\frac{\gamma}{\beta})\sqrt{\alpha^2-\beta}, & k=2l,\ l\geq 0, \\
\alpha(\theta_0-\frac{\gamma}{\beta})+\theta_0^\prime, & k=2l-1,\ l\geq 1.
\end{cases}
\end{equation}
This shows that $\theta$ has extension for $\alpha^2-\beta=0$ as well by the same analytical expression
as in \eqref{e:2odesolseries}. It remains to show that, for $\alpha^2-\beta=0$ this extension coincides with
the analytical expression provided by \eqref{e:2odesoldisnula}. To see this, we use the fact that the power
series expansion of the exponential function is uniformly convergent on any compact set and hence, when
letting $\alpha^2-\beta\ra 0$ in \eqref{e:2odesolseries} we get
\begin{align}
\lim_{\alpha^2-\beta\ra 0}\!y(t)\!= & 
\frac{\gamma}{\alpha^2}-\frac{dp}{\alpha^2}\bigl[H(t-m)-H(t-n)\bigr] \\
& \
+\emath^{-\alpha t}\biggl[ 
(\theta_0\!-\!\frac{\gamma}{\alpha^2})\!+\!\bigl(v\!-\!(\frac{\gamma}{\alpha^2}\!-\!u)\alpha\bigr)t
-ep(m\!-\!t)H(m\!-\!t)\emath^{m\alpha}\!+\!ep(n\!-\!t)H(n\!-\!t)\emath^{n\alpha}
\biggl] .\nonumber 
\end{align}

Finally, from \eqref{e:2odesolseries} and taking into account once more that the power series expansion
of the exponential function is uniformly continuous on any compact set, it follows that
the solution $y$ is uniformly
continuous on compact subsets with respect to the parameters $b,d,e,g,m,n>0$, $m<n$, 
$a,c,\theta_0,\theta_0^\prime\in\mathbb{R}$
and $t\in\RR_+\setminus\{m,n\}$.
\end{proof}

Theorem~\ref{t:robust} says that the analytical form provided at \eqref{e:2odesola} is robust, in particular, the limit when
$\alpha^2-\beta\ra0$ is the function $\theta$ provided at \eqref{e:2odesoldisnula}.
In addition, if $\alpha^2-\beta<0$ then, letting $\omega=\sqrt{\beta-\alpha^2}$, the solution $\theta$ defined
as in \eqref{e:2odesola} has the representation \eqref{e:2odesolcoma}.

\subsection{The General Model.}\label{ss:gm}
We obtain the general model of the evolution of certainty, in the presence of a forcing factor, by aggregating
the subjective certainty and the objective certainty. The parameters of the objective certainty will
be denoted $a_1,b_1,c_1,d_1,\theta_{0,1}$ and hence the objective certainty $\theta_1$, see 
\eqref{e:thetaunuar},
\begin{align}
\theta_1(t) & = \frac{a_1-c_1}{b_1+d_1}- \frac{d_1 p}{b_1+d_1} \bigl(\heaviside(t - m) 
- \heaviside(t-n)\bigr) \label{e:thetaunuarunu}\\
& \ \ \ -\emath^{-t(b_1+d_1)} \biggl[\bigl(\frac{a_1-c_1}{b_1+d_1}-\theta_{0,1}\bigr) - \frac{d_1p}{b_1+d_1} 
\biggl(\heaviside(t - m) \emath^{(b_1 + d_1) m} 
-  \heaviside(t - n) \emath^{(b_1  + d_1) n}\biggr)\biggr]. \nonumber
 \end{align}
In order to obtain the general model for the evolution of certainty, 
we aggregate the subjective certainty, as provided by the 
formula \eqref{e:2odesola} in such a way that the objective certainty, as provided by \eqref{e:thetaunuar}, 
plays the role of the coupling function. For consistency reasons, we first introduce the following notation:
\begin{equation}
\beta_1=b_1+d_1,\quad \gamma_1=a_1-c_1.
\end{equation}
Then, with notation as in Subsection~\ref{ss:mmsc}, the general model of the evolution function $\theta$
of certainty is: if $\alpha^2-\beta\neq 0$ then
\begin{align}
\theta(t) & =
\frac{\gamma_1}{\beta_1}- \frac{d_1 p}{\beta_1} \bigl(\heaviside(t - m) 
- \heaviside(t-n)\bigr)
 -\emath^{-t\beta_1} \biggl[\bigl(\frac{\gamma_1}{\beta_1}-\theta_{0,1}\bigr) - \frac{d_1p}{\beta_1} 
\biggl(\heaviside(t - m) \emath^{\beta_1 m} 
-  \heaviside(t - n) \emath^{\beta_1 n}\biggr)\biggr] \label{e:thetagenerala}\\
& \ \  +\! \frac{\emath^{-\alpha t}}{2\sqrt{\alpha^2\!-\!\beta}}
\biggl[ \bigl((\theta_0\!-\!\frac{\gamma}{\beta})(\alpha\!+\!\sqrt{\alpha^2\!-\!\beta})+\theta_0^\prime\bigr)  
\emath^{t\sqrt{\alpha^2\!-\!\beta}} \nonumber 
+ \bigl((\theta_0\!-\!\frac{\gamma}{\beta})(-\alpha\!+\!\sqrt{\alpha^2\!-\!\beta})\!-\!\theta_0^\prime\bigr)
\emath^{-t\sqrt{\alpha^2\!-\!\beta}}\biggr]\nonumber \\
& \ \ +\! \frac{ep\emath^{-\alpha t}}{2\sqrt{\alpha^2\!-\!\beta}}\biggl[H(t\!-\!m)\emath^{m\alpha}
\bigl(\emath^{(t-m)\sqrt{\alpha^2\!-\!\beta}}\!-\!\emath^{(m-t)\sqrt{\alpha^2\!-\!\beta}}\bigr)
 - H(t\!-\!n)\emath^{n\alpha}
\bigl(\emath^{(t-n)\sqrt{\alpha^2\!-\!\beta}}\!-\!\emath^{(n-t)\sqrt{\alpha^2\!-\!\beta}}\bigr)\biggr],\nonumber
\end{align}
and, if $\alpha^2-\beta=0$ then
\begin{align}
\theta(t) & =
\frac{\gamma_1}{\beta_1}- \frac{d_1 p}{\beta_1} \bigl(\heaviside(t - m) 
- \heaviside(t-n)\bigr)
 -\emath^{-t\beta_1} \biggl[\bigl(\frac{\gamma_1}{\beta_1}-\theta_{0,1}\bigr) - \frac{d_1p}{\beta_1} 
\biggl(\heaviside(t - m) \emath^{\beta_1 m} 
-  \heaviside(t - n) \emath^{\beta_1 n}\biggr)\biggr] \label{e:thetageneraladisnul}\\
& \phantom{\frac{\gamma}{\alpha^2}+}
+\emath^{-\alpha t}\biggl[ 
(\theta_0-\frac{\gamma}{\alpha^2})+\bigl(\theta_0^\prime-(\frac{\gamma}{\alpha^2}-\theta_0)\alpha\bigr)t
-ep(m-t)H(m-t)\emath^{m\alpha}+ep(n-t)H(n-t)\emath^{n\alpha}
\biggl] .\nonumber
\end{align}

We also point out that, in case $\alpha^2-\beta<0$, letting $\omega=\sqrt{\beta-\alpha^2}$ then the formula
in \eqref{e:thetagenerala} has the equivalent representation
\begin{align}
\theta(t) & =
\frac{\gamma_1}{\beta_1}- \frac{d_1 p}{\beta_1} \bigl(\heaviside(t - m) 
- \heaviside(t-n)\bigr)
 -\emath^{-t\beta_1} \biggl[\bigl(\frac{\gamma_1}{\beta_1}-\theta_{0,1}\bigr) - \frac{d_1p}{\beta_1} 
\biggl(\heaviside(t - m) \emath^{\beta_1 m} 
-  \heaviside(t - n) \emath^{\beta_1 n}\biggr)\biggr] \label{e:thetageneraladidneg}\\
& \ \ \ +\emath^{-\alpha t} \times \nonumber \\
& \ \ \ \ \times\! \biggl[\! \biggl( (\theta_0 -\frac{\gamma}{\beta})\!-\!\frac{ep}{\omega}\bigl[ \emath^{m\omega}
\sin(m\omega) H(t-m)-\emath^{n\omega}\sin(n\omega)H(t-n)\bigr]\!\biggr)\!\cos(\omega t) \nonumber \\
& \ \ \ \ \ \ \ \ +\biggl(\frac{\theta_0^\prime +\alpha(\theta_0 -\frac{\gamma}{\beta})}{\omega}  
+\frac{ep}{\omega}\bigl[ \emath^{m\omega}\cos(m\omega)H(t-m)-\emath^{n\omega}
\cos(n\omega)H(t-n)\bigr]\biggr)\sin(\omega t)\biggr].\nonumber
\end{align}

Theorem~\ref{t:robust} says the analytical expression in \eqref{e:thetagenerala}
is robust and can be used for computer simulation and approximation by the method of least squares of 
experimental data. Here we make some computer simulation by using the general formula
\eqref{e:thetagenerala}. We first point out the effect of the perturbation function $f$ on the evolution of certainty
when the parameter $\alpha$, that has the interpretation of sensitivity to the rate of change of certainty, varies 
with respect to the bifurcation that is controlled by the sign of the discriminant $\alpha^2-\beta$, see 
Figure~\ref{f:simulation3d}. The perturbation occurs in the time interval $[5,10]$ and we observe the
change of the evolution of certainty as a function of time for all possible values of $\alpha$ and all possible
signs of the discriminant $\alpha^2-\beta$, although this perturbation yields changes of different shapes. The
bifurcation occurs at $\alpha=\sqrt{0.5}=0.707106781186544\ldots$ which explains the ridge in the 
neighbourhood of that value.

\begin{figure}[h]
\includegraphics[width=12cm]{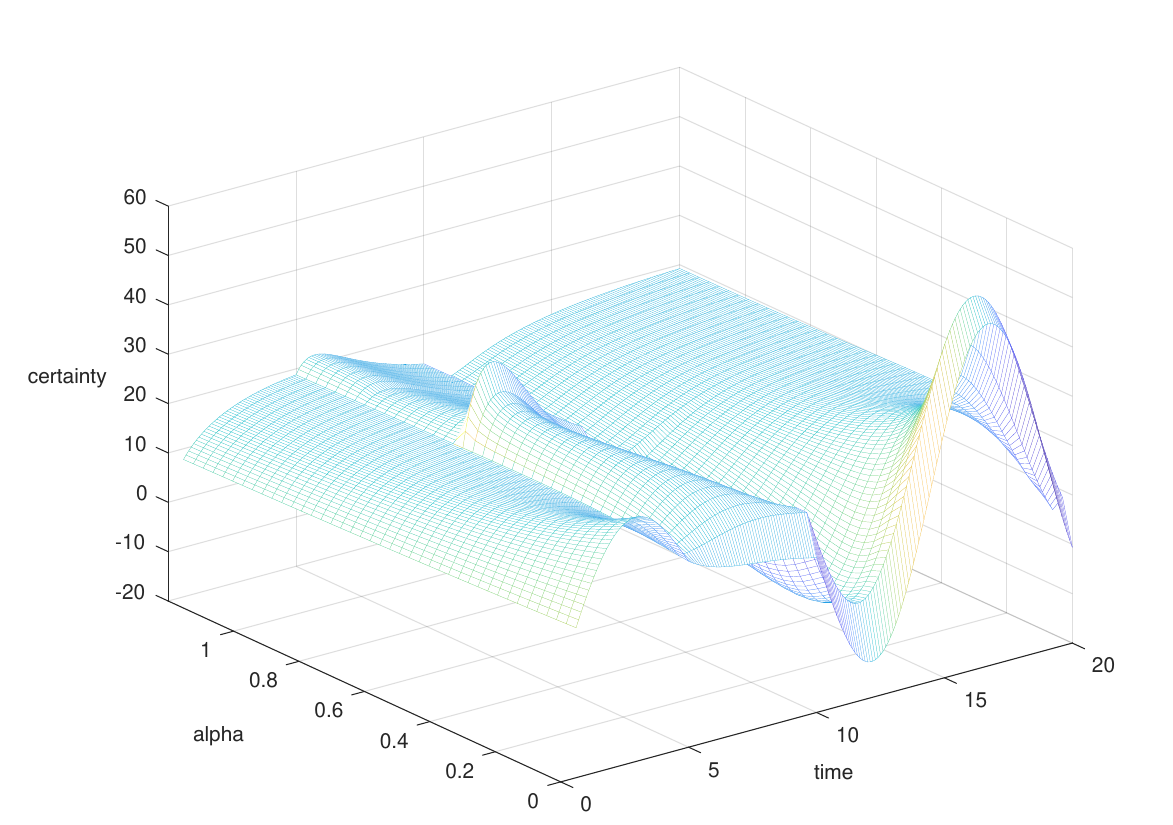}
\caption{In this simulation, we use the parameters, with respect to the notation as in \eqref{e:thetagenerala},
as follows: $\beta=0.5$, $\gamma=2.02$, $\theta_0=0$, $\theta_0^\prime=-2.2$, $\beta_1=0.5$, 
$\gamma_1=10$, $\theta_{0,1}=5$, $d_1=0.25$, $e=41$, $m=5$, $n=10$, $p=20$, while the parameter 
$\alpha$ varies in the interval $[0,1.2]$ with the sampling step $0.03$.}
\label{f:simulation3d}
\end{figure}

To be more precise, in Figure~\ref{f:sections} we depict the evolution of certainty under the perturbation of the 
function $f$ in three relevant cases: when $\alpha^2-\beta<0$ and we have a permanent oscillation of
frequency $\omega\neq 0$, and when $\alpha^2-\beta=0$ or $\alpha^2-\beta>0$ and we have no permanent 
oscillation, by emphasising the effect of the perturbation. Compared to Figure~\ref{f:simulation3d}, the three 
graphs are sections of the mesh for the corresponding values of $\alpha$, in continuous line for $0.03$,
in dashed line for $0.69$, and in dotted line for $1.2$.
In this simulation, the perturbation occurs in
the time interval $[5,10]$ when the evolution of certainty is significantly affected. We observe that, the
effect of the perturbation has a bigger amplitude in the case of negative or close to zero discriminant and 
relatively small in the case of positive discriminant. In all three cases, it needs some time to fade away
the effect of this perturbation. The value $\alpha=0.69$ is very close to the
bifurcation that appears at $\alpha^2-\beta=0$, in our case $\beta=0.5$ and 
$\alpha=0.707106781186544\ldots$, which explains the shape of the dashed line in the interval $[5,10]$, due
to the singularity in the neighbourhood of the bifurcation.

\begin{figure}[h]
\includegraphics[width=9cm]{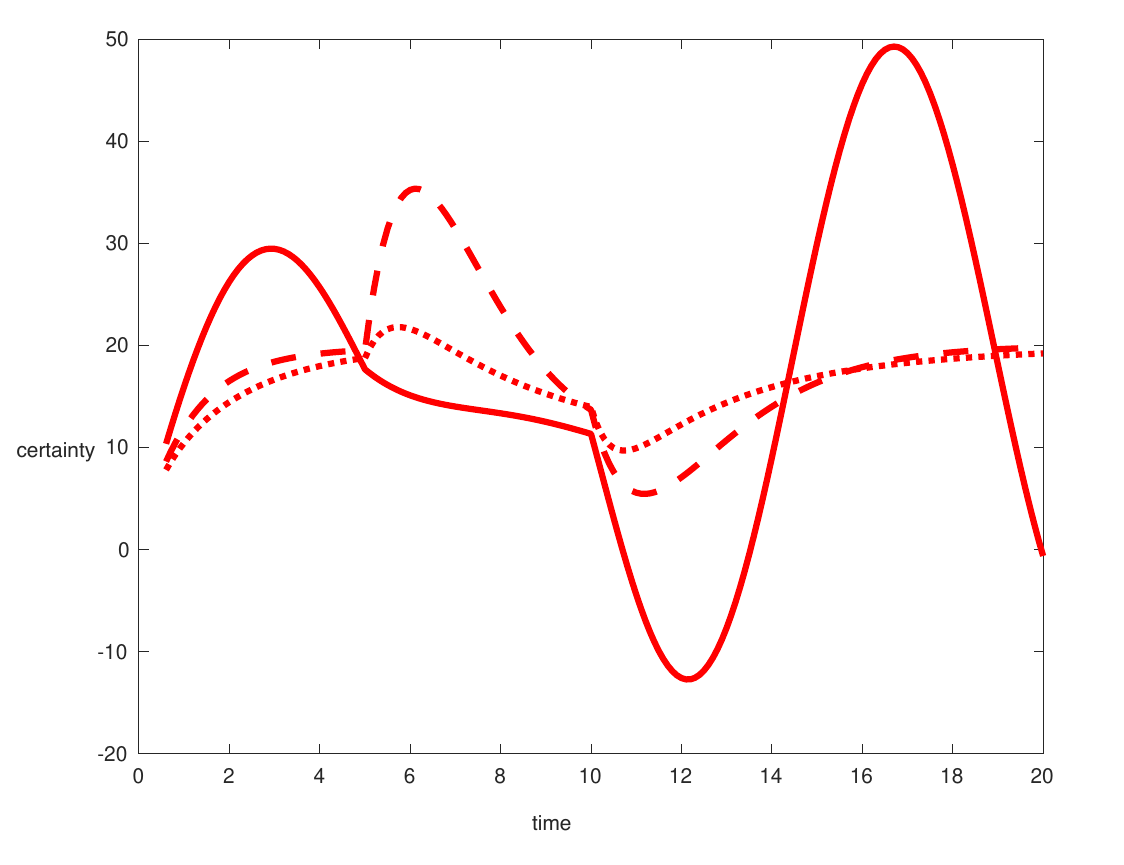}
\caption{In this simulation, we have three sections of the mesh from Figure~\ref{f:simulation3d}, corresponding
to $\alpha=0.03$, the graph in continuous line, to $\alpha=0.69$, the graph in dashed line, and to 
$\alpha=1.2$, the graph in dotted line.}
\label{f:sections}
\end{figure}

\section{Data Analysis}\label{s:da}

The purpose of the data analysis is to test whether the model can identify a forcing factor in the EoC for the participants in a group decision-making task aimed to reach consensus, on the same data as those given by Gheondea-Eladi (2016), where the group discussion is assumed to not exert a forcing factor. In this section we will provide the methodological description of the data analysis employed to this end. The analysis to be presented aims to test whether the model identified in section~\ref{ss:gm} differentiates between the cases where a forcing factor appears and those where it does not. To this end the data analysis will check: 1) whether there is a forcing factor (visible through oscillatory behaviour); 2) when does the forcing begin and end within the group discussion; 3) whether the equilibrium level of the certainty compared to the case in which a forcing factor was not modelled increases or decreases. Future research should document clearly forced and clearly non-forced group discussions in order to test and later on categorize the depth of the differentiation provided by this model.

\subsection{Data collection}
Ninety-seven students of sociology and social work, were asked to complete the NASA  moon survival task (Teleometrics International, 2007) in groups of 4-6 people. 

\emph{You are a member of a crew scheduled to meet the mother-ship on the bright face of the Moon. Due to some malfunction your ship was forced to land 322 kilometers away from the meeting point. During the landing procedure, most of the equipment on board was destroyed. Since your survival depends upon reaching the mother-ship, you have to choose the most important items available in order to walk the  distance to the mother-ship. There are 15 items left intact after the impact. Your mission is to order them according to their importance such that your crew will reach the meeting point. Number with 1 the most important item and with 10 the least important item.
\begin{itemize}
\item[–] Match box
\item[–] Condensed food
\item[–] 15 meters of nylon thread
\item[–] Parachute silk
\item[–] Portable heating unit
\item[–] Two 45 caliber pistols
\item[–] 1 box of condensed milk
\item[–] 2 oxygen tanks of 45 kg each
\item[–] Stellar map (around the moon)
\item[–] Automatic inflation rescue vest
\item[–] Magnetic compass
\item[–] 19 liters of water
\item[–] Signaling missiles
\item[–] First aid kit including needles and a syringe
\item[–] Radio emitter and transmitter with solar batteries
\end{itemize}}
All students have been verbally informed about the content of the experiment and that their participation is voluntary, that they may choose to leave the experiment at any time with no explanation and that all information they provide is annonymous. Students chose nicknames for themselves on the answering sheets and no recording between their names and nicknames has been collected. Those who participated were invited for educational purposes and did not receive any kind of course credit or reward. Each individual in the group first indicated their own solution, as an individual hierarchy and its confidence level and then discussed it with the 
group in order to achieve consensus on a group hierachy. Participants were also asked to provide the individual certainty level for the group solution.  During the group discussion, each participant wrote when each perceived change in the individual solution occured, as follows: the new solutions, the time at which each new solution occurred, the confidence level of the new solution itself and the confidence of the new solution compared to the initial individual one. Participants were asked to report these changes as soon as they perceive them. After the group discussion had reached a consensus or when consensus was considered to be impossible, each participant was asked to fill out the following information: 
individually write the group solution, write how certain she/he was about the group decision, write how satisfied they were with the way the decision was made, and how interesting they found the decision problem. 

Participants' experimental output comprised of two time series, one corresponding to the confidence of the new 
decision with respect to the initial individual decision (the reference point salient confidence measurement, called CII) and the other corresponding to the confidence of the new solution itself (the common confidence measurement, called CIF). This double measurement choice was undertaken in order to solve the following issue: participants in the experimental pretest reported the confidence evaluations with different reference points (for example, some reported the confidence level of the third observed change in comparison to the second observed change and the fourth in comparison to the third, while others were not fully aware which reference point they used). To solve this, we adapted the procedure proposed by Hoch (1987), cited in~(Stanovich and West, 1998) for the "certainty effect" and used the CII and CIF to derive, by subtraction, the \emph{absolute individual certainty} (abreviated CI). This procedure is explained in detail in Section 2.3 
of~(Gheondea-Eladi, 2016).  At each moment in time during the group discussion there were two direct measures of confidence: 1) CII was measured through the following question: "How certain are you (now) of the initial hierarchy?"; 2) CIF was measured through the question: "How certain are you (now) of the new hierarchy?". Both CII and CIF are measured repeatedly during the group discussion.  The differences between the two measures were verbally explained and exemplified in the experimental instructions though the following explanation: 1) for CII: at this time and after the discussion you had so far, how certain are you of the initial hierarchy; 2) for CIF: at this time and after the discussion you had so far, how certain are you of the new hierarchy you have just indicated. Based upon post-experimental discussions with the participants  two other experimental tasks and phrasings were tested for these questions before reaching the final form.

%\begin{figure}[thb]\label{f:DRM}
%\includegraphics[width=10cm]{C19New2.pdf}
%\caption{Participant 1.} {Data for CI (points), CII (circle), \c si CIF (star) and the theoretical model curve fit for each.} \label{top}
%\end{figure}
%
%\begin{figure}[h]\label{f:C}
%\includegraphics[width=10cm]{RM12New2.pdf}
%\caption{Participant 2.} {Data for CI (points), CII (circle), \c si CIF (star) and the theoretical model curve fit for each} \label{top}
%\end{figure}

For the model to be tested it was important that the communication in the group discussion could be approximated by time, and consequently required a single information unit to be communicated at a time. The current sample includes participants who indicated one change at a time and those who indicated more than one change at a time only once or twice. In the latter case, the reported multiple changes have been removed from the time series without much loss of information. In contrast to the previous papers on this topic (Gheondea-Eladi, 2016; Gheondea-Eladi and Gheondea, 2022), we have also discarded the instances in the time series where one of the measurements (CII or CIF) had missing data, thus leaving the size of the time series smaller than in the previous case (from a maximum of $N=20$ to $15$). Since in the previous paper, the general model included seven coefficients and the present one included thirteen, three of the participants' data were not enough for the Levenberg-Marquardt algorithm to handle bound constrainsts and for the trust-region-reflective algorithm to have at least as many equations as variables. As a consequence, the sample includes 57 participants.

% I forgot to analyse the data from 3 participants...In the second paper the sample is 60 and here 57. I only realised now...
%However, 21 of the participants who indicated more than one change at a time reported this only once or twice. In such cases, without much loss of information, these instances could be removed from the time series, leaving an individual time series sample of 60 individuals in the second paper \cite(Gheondea-Eladi2022}.  This facilitated the analysis of all measurements (CII, CIF and CI) in a single instance, as opposed to separately, as performed previously, where the algorithm required the same length of the time series for all the three measurements.

\subsection{Numerical Procedure}
\label{np}
Each participant yielded a time series, represented by two sequences of data:

\begin{equation}
xdata=(x_1, x_2, x_3,..., x_N),\quad ydata=(y_1, y_2, y_3, ..., y_N),
\end{equation}
where $N$ is the time series size, $x_j$ is the time sequence expressed in seconds, and $y_j$ is the reported level of certainty, corresponding to time $x_j$. The size of the time series, $N$ is between $11$ and $16$ and $y_j$ has values between $-7$ and $7$ for CI and between $0$ and $7$ for CII and CIF.

The least squares approximation method was employed, under the  conditions required for time series 
(i.e.\ error is spherical according to~(Beck2007a) 
by using \emph{lsqcurvefit} in \textsc{MATLAB}. As in the
previous analysis (Gheondea-Eladi. 2016), the Euclidean distance between the general model in \eqref{e:thetagenerala} 
and the data series $\mathbb{R}^N $ is minimised. 

The algorithm (available in the Supplementary Material) searches for a vector $c=(c_j)_{j=1}^{13}$, within a confidence 
interval $[c_j^\mathrm{min}, c_j^\mathrm{max}]$, for each $j=1,\ldots,13$. 

\subsection{Data Analysis Procedure}
\label{da:p}
The data analysis procedure is designed to check: 
1) whether there is a forcing factor (visible through oscillatory behaviour); 2) when does the forcing begin and end within the group discussion; 3) whether there is an increase or a decrease in the equilibrium level of the certainty compared to the case in which a forcing factor was not modelled.

In the model identified in Section~\ref{ss:gm} the value of $p$ indicates whether there is a forcing factor or not. If $p$ is big enough, then the forcing factor can be identified. If $p$ is small, then the forcing factor either did not have an influence or there was no forcing factor. In the same model, we can understand at what moment in time the forcing factor began by looking at the value of $n$ and when it has ended (if at all), by looking at the value of $m$. 
In order to check if the equilibrium value of the certainty (the certainty after the group discussion) is different or the same as when there is no forcing factor, we will analyse the modeled equilibrium values and those from the original paper of (Gheondea-Eladi, 2016), using the same data. In the analysis performed here we will see that in many cases the predicted equilibrium value of the certainty is smaller than in the case when there was no forcing factor taken into consideration. The theoretical reasons amenable to this state of affairs have been reviewed in the Introduction.

%Observam ca daca exista factorul de fotare, atunci certitudinea se echilibreaza la o valoare mai mica decat valuarea normala (daca ar fi neperturbat) (din articolul trecut). Sa scot valoarea de echilibru pentru fiecare participant in parte din articolul trecut. 

%m si n sunt punctele de timp in care incepe si respectiv se termina influenta factorului de fortare. M din date este intre 20-30 de minute, dar asta coincide cu momentul in care participantii incep sa dea datele.

 %Distanta solutiei spune cat de departe e solutia gasita de algoritmul lsqfit fata de shooting data. Eroarea relativa spune cat de bine aproximam prin curba calculata. Daca eroare relativa este mai aproape de 0, atunci curba teoretica trece exact prin punctul experimental. 

\subsection{Empirical Results}
The data analysis procedure is designed to check: 1) whether 
there is a forcing factor (visible through oscillatory behaviour); 2) when does the forcing begin 
and end within the group discussion; 3) whether there is an increase or a decrease in the 
equilibrium level of the confidence compared to the case in which a forcing factor was not 
modelled. 

In the general model, the value of $p$ indicates whether there is a forcing factor or 
not. If $p$ is big enough, then the forcing factor can be identified. If $p$ is small, then the forcing 
factor either did not have an influence or there was no forcing factor. Further research should establish lower and upper boundaries in the values of $p$, in order to be able to differentiate between the influence of different forcing factors. In the same model, we can 
understand at what moment in time the forcing factor began by looking at the value of $n$ and 
when it has ended (if at all), by looking at the value of $m$. In order to check if the equilibrium 
value of the confidence (the confidence after the group discussion) is different or the same as when 
there is no forcing factor, we will analyse the modelled equilibrium values and those from the 
original paper of (1), using the same data.

\subsubsection{Is there an observable forcing factor?}
Visible oscillations (for $p>1$) occur in 39 of 57 cases. In Figure~\ref{da:f2}, A, we provide a histogram of the values of $p$, which reflects the amplitude of the oscillation of the forcing factor, for the confidence measured as CI (see the Experimental Design section). Reference values of $p$ may be established for different forcing factors (such as when a participant artificially increases the confidence of another by boosting self-confidence, when the decision is deferred to an authoritative figure or when there is a motivational movie running for the group, and so on). Note that the current experiment has been constructed with instruction to reach consensus and no forcing factor manipulation has been applied to the group. Thus it serves as the base rate for future experiments aimed at estimating model parameters for specific types of forcing factors. 
In evaluating the coefficients $p, m, n$ we assess the quality of the approximation algorithm by minimizing the relative error (see Table~\ref{da:ff}) and the distance to the solution. The latter reflects how much the algorithm improved our initial guess over the coefficients and it should only be viewed as a measure regarding the algorithms’ initial conditions. 

\begin{table}
\caption{Measures of central tendency of relevant model coefficients: $p$ = oscillation coefficient; $m$ = starting point of forcing factor; $n$ = end point of forcing factor.}
\label{da:ff}
\begin{tabular}{c|lllll|ll}

& Mean		&		Median		 &	Std.Dev. & Min. & Max & Unit	\\	

\hline
Relative error & 0.645	& 0.629	& 0.443 &0.004 & 3.749 & \\
Distance to the solution & 3.969 &	3.999 &	1.914 & 0.083 & 9.211 & \\
$p$ & 1.949	& 1.306 &	1.826 & & & \\
$m$ & 29.957 &	29.928 & 	8.915 & &  & 	seconds\\
$n$ & 100 &	100 &	0	& & & seconds \\
\hline

\end{tabular}
\end{table}

%\begin{figure}\label{da:p1}
%\includegraphics[width=5cm]{P_forCII.png}
%\caption{Histogram of the oscillation coefficient, $p$ for certainty measured as CII.}
%\end{figure}
%
%\begin{figure}\label{da:p2}
%\includegraphics[width=5cm]{P_forCIF.png}
%\caption{Histogram of the oscillation coefficient, $p$, for certainty measured as CIF.}
%\end{figure}

\begin{figure}
\includegraphics[width=13cm]{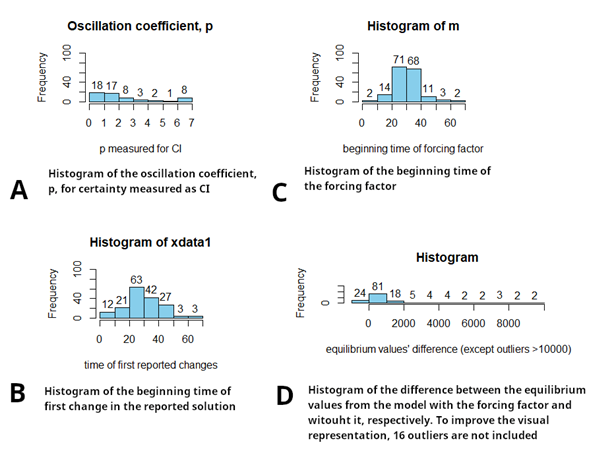}
\caption{Results.}\label{da:f2}
\end{figure}

%\begin{figure}\label{da:p3}
%\includegraphics[width=5cm]{P_forCI.png}
%\caption{Histogram of the oscillation coefficient, $p$ for certainty measured as CI.}
%\end{figure}

\subsubsection{When does the forcing factor begin and end?}
Irrespective of how the confidence was measured, the forcing factor ends at the end of the discussion (n=100) and it begins  approximately between 20 and 40 minutes after the group discussion begins for 139 of the measurements (approx. 81\% of $57 \times 3 = 171$ measurements) (see $m$ in Table 1 and Figure~\ref{da:f2}, C). Around this time, participants also begin to report changes in their initial solution (Figure~\ref{da:f2}, B). 

%\begin{figure}\label{da:m}
%\includegraphics[width=5cm]{m.png}
%\caption{Histogram of the begining time of forcing factor.}
%\end{figure}
%
%\begin{figure}\label{da:xdata1}
%\includegraphics[width=5cm]{xdata1.png}
%\caption{Histogram of the begining time of first change in the reported solution.}
%\end{figure}

\subsubsection{What happens to the equilibrium value of the confidence when a forcing factor appears?}

Here we compare the equilibrium values predicted by the model without a forcing factor (based on the values presented in (Gheondea-Eladi and Gheondea, 2022) and those predicted by the model with a forcing factor, given in eq. (1). For the model which includes the forcing factor, \eqref{e:thetageneraladidneg}, the equilibrium value is given by the values of $\frac{\gamma_1}{\beta_1}- \frac{d_1 p}{\beta_1}$, when the Heaviside function is constant and equal to $1$. In the database, available in the Supplementary Material, these correspond to coefficients $\gamma_1=c_7, \beta_1=c_7, d_1=c_9$ and $p=c_{13}$ (see Section~\ref{np}). For the model without the forcing factor, based on equation (17) from (Gheondea-Eladi and Gheondea, 2022), the equilibrium value is given by the values of $\frac{\gamma_1}{\beta_1}$ and in the database corresponding to this paper, these are $\gamma_1=c_7$ and $\beta_1=c_6$. 

To do this, we employ a subset of participants common to both manipulations. We find that the equilibrium value for the model with the forcing factor is lower than that without it: $24/165$, ($85.45\%$ of measurements) lead to such cases, (Figure~\ref{da:f2}, D). We consider that this is one of the main results of this paper. Generally, this situation may be justified by cases in which the decision is delegated to some extent. 

Methodologically, in order to compare the equilibrium value of the confidence in the model with the forcing factor and that of the model without it, we perform the following steps:
\begin{enumerate}
\item[1.] We compare the equilibrium values from the two models (with and without the forcing 
factor), by looking at the distribution of differences between the equilibrium values 
(EQV) of the two models and observing the cases when the difference between the EQV 
with the forcing factor and the EQV without it, is negative.
\item[2.] We perfom a robustness check for the two algorithms, which means that we look at whether the coefficients obtained through the two algorithms are similar. We find that any differences observed are likely to be due to error (sig.$=0.810$ for $\beta_1$ and sig.$=0.201$ for $\gamma_1$, paired samples t-tests). 

\end{enumerate}

%
%
%\begin{table}
%\caption{Equilibrium values based on the model with and without a forcing factor}
%\label{da:eqval}
%\begin{tabular}{c|lll}
%
%Model & 	Mean	&		Median		 &		Std. Deviation	\\	
%\hline
%Without forcing factor (all measures)  &481.533	 & 321.105	 &	467.733 \\
%Without forcing factor (CI) & 186.198 & 0.8561 & 323.9304 \\
%\hline
%With forcing factor (all measures) &	-6716.0  & 	-10.8 &	34833.71  \\
%With forcing factor (CI) & -1803.193 & -12.006 & 5109.577 \\
%\hline
%
%\end{tabular}
%\end{table}

\section{Conclusions}
In this article, we obtained two notable results. First, the confidence reflected by the model with the forcing factor has a lower equilibrium value than in the model without. Second, we show that, when a forcing factor appears, it can be expected to induce additional oscillations of the confidence, but these oscillations are not damped. 

Prolonged oscillation of the certainty level can be difficult, unpleasant or costly (Zamfir, 2005), but so could the prospect of acting with lower confidence (Punchochar and Fox, 2004).  Given the interpretations of a forcing factor as either boosts of self-confidence, deferral of decision or other kinds of influences which affect confidence that are not related to information (see Introduction), this result raises a fundamental question for decision coaching and group decision-making practice: under which conditions is it useful to make a decision with highest degree of confidence and when not? And also: how high (by allowing the confidence to oscillate during the process) or low (by, for example, deferring it to a certain principle or path/idea) do we want the confidence level to be? In our model, the sensitivity to the absolute value of the confidence and the sensitivity to the rate of change (Gheondea-Eladi and Gheondea, 2022) are used to account for this. 

The second important finding of this paper indicates that we identify forced decisions through the presence of non-attenuated oscillations in the EoC during group decisions. Crucially, the presence of non-attenuated oscillations serves as an analytical marker for identifying "forced" decisions, allowing researchers to differentiate between normative information integration and artificial elevations in certainty. This versatile framework establishes a foundation for further investigation into various modes of social influence and their corresponding technological applications.

\subsection{Limitations and extensions}
	There are two main aspects to discuss for the applicability of the model. First, although this model is tested on a decision with a known solution, its applicability is not restricted to such cases. This is due to the fact that the assumptions of the mathematical model employed do not rely on the existence of a solution to the decision, nor do the empirical measurements of the confidence. Second, the model is suited for small groups, but not for large ones or networks, because it assumes that each group member is able to interact with the others. However relying on a differential equation model paves the way towards embedding it into social network models.

	There is one important limitation of this model. The results are only valid for the cases when the confidence function is a linear representation of time. This means, in practical terms that the results are valid only for those people who were able to differentiate accurately when a change has happened within a unit of time. For people who reported more than one change in the hierarchy at a given time, a univalued function of time cannot be used. For these people multivalued functions (relations) would be necessary. 

\subsection{Technical Considerations}
	An advantage of using step functions for the forcing factor (which have a finite number of discontinuities and hence a finite number of singularities and generalised functions (distributions), which have a low number of paramenters) is that it allows us to compare the EoC with a model which takes into consideration the presence of forcing factors even though the group discussion aimed at reaching consensus is said to not lead to deliberate interventions in the level of the confidence.  A disadvantage of using the Heaviside functions for modelling the perturbation is that it has singularities at the starting and ending points and hence, robustness holds only on compact intervals that avoid these points. This can be remedied if smooth approximations of the Heaviside functions will be used in the model, for example the flexible sigmoid functions, but this requires more parameters in the minimisation problem and hence a different experiment with more sampling points. This will be the subject of a follow-up of this research.
	\bigskip
	
%\textbf{Acknowledgements:}
%This research was not funded from any source.
%
%This research did not use any AI software at any stage.
%
%The authors declare no conflict of interest.
\bigskip

\setcounter{equation}{0}
\renewcommand\theequation{A.\arabic{equation}}
%%%%%%

\section*{Appendix: Proof of Theorem~\ref{t:2ode}}

\subsection*{A.1: The 2nd Order ODE with Forcing Factor.}
We consider the second order initial value problem (2nd order IVP)
\begin{align}\label{e:2ode}
y^{\prime\prime}(t) +(e+g)y^\prime(t)+ (b+d)y(t) & =(a-c)-dp(H(t-m)-H(t-n)) \\ & \phantom{(a-c)+} 
+ep(\delta(t-m)-\delta(t-n)), 
\nonumber \\ y(0)  = u, \quad & \quad
y^\prime(0) = v, \nonumber
\end{align}
under the assumption that all parameters $a,b,c,d,e,g$ are positive, while $u$ and $v$ are real.
Here $H$ denotes the Heaviside function, $\delta$ is the Dirac "function", and $0<m<n$. 
If $p=0$ we have the unperturbed 2nd order IVP and
if $p\neq 0$ we have the perturbed 2nd order IVP. The perturbation part is given by the function 
$df(t)=dp(H(t-m)-H(t-n))$, where $f(t)=p(H(t-m)-H(t-n))$ is the forcing factor.

In order to solve the perturbed linear 2nd order IVP \eqref{e:2ode} we use {\tt dsolve} program of MATLAB,
that requires 
\begin{equation}(e+g)^2-4(b+d)\neq 0, \quad \lambda_{1,2}=\frac{e+g\pm \sqrt{(e+g)^2-4(b+d)}}{2}\neq 0,
\end{equation}
and provides the solution
\begin{align}
y(t) & = \exp(-t (e/2 + g/2 -  \sqrt{(e+g)^2-4(b+d)}/2))\times \\
& \ \ \ 
\times\biggl(2 a \frac{\exp((e t)/2 + (g t)/2 - (t  \sqrt{(e+g)^2-4(b+d)})/2)}{(e + g -  \sqrt{(e+g)^2-4(b+d)})  \sqrt{(e+g)^2-4(b+d)}} \nonumber \\
& \ \ \ \ \  \ \ \ \ \ - 2 c \frac{\exp((e t)/2 + (g t)/2 - (t  \sqrt{(e+g)^2-4(b+d)})/2)}{(e + g -  \sqrt{(e+g)^2-4(b+d)})  \sqrt{(e+g)^2-4(b+d)}}  \nonumber \\
&  \ \ \ \ \  \ \ \ \ \ - e p \heaviside(t - m) \frac{\exp((e m)/2 + (g m)/2 - (m  \sqrt{(e+g)^2-4(b+d)})/2)}{ \sqrt{(e+g)^2-4(b+d)}}  \nonumber \\
&  \ \ \ \ \  \ \ \ \ \ + e p \heaviside(t - n) \frac{\exp((e n)/2 + (g n)/2 - (n  \sqrt{(e+g)^2-4(b+d)})/2)}{ \sqrt{(e+g)^2-4(b+d)}}  \nonumber \\
&  \ \ \ \ \  \ \ \ \ \ + d p \frac{\exp((e t)/2 + (g t)/2 - (t  \sqrt{(e+g)^2-4(b+d)})/2) (\sign(m - t) - 1)}{(e + g -  \sqrt{(e+g)^2-4(b+d)})  \sqrt{(e+g)^2-4(b+d)}}  \nonumber \\
&  \ \ \ \ \  \ \ \ \ \ - d p \frac{\exp((e t)/2 + (g t)/2 - (t  \sqrt{(e+g)^2-4(b+d)})/2) (\sign(n - t) - 1)}{(e + g -  \sqrt{(e+g)^2-4(b+d)})  \sqrt{(e+g)^2-4(b+d)}}\biggr)  \nonumber  \\
%%%%%%%
& \ \ \ - \exp(-t (e/2 + g/2 +  \sqrt{(e+g)^2-4(b+d)}/2)) \times  \nonumber \\
& \ \ \
\times \biggl(2 a \frac{\exp((e t)/2 + (g t)/2 + (t  \sqrt{(e+g)^2-4(b+d)})/2)}{(e + g +  \sqrt{(e+g)^2-4(b+d)})  \sqrt{(e+g)^2-4(b+d)}}  \nonumber \\
&  \ \ \ \ \  \ \ \ \ \ - 2 c \frac{\exp((e t)/2 + (g t)/2 + (t  \sqrt{(e+g)^2-4(b+d)})/2)}{(e + g +  \sqrt{(e+g)^2-4(b+d)})  \sqrt{(e+g)^2-4(b+d)}}  \nonumber \\
&  \ \ \ \ \  \ \ \ \ \ - e p \heaviside(t - m) \frac{\exp((e m)/2 + (g m)/2 + (m  \sqrt{(e+g)^2-4(b+d)})/2)}{ \sqrt{(e+g)^2-4(b+d)}}  \nonumber  \\
&  \ \ \ \ \  \ \ \ \ \ + e p \heaviside(t - n) \frac{\exp((e n)/2 + (g n)/2 + (n  \sqrt{(e+g)^2-4(b+d)})/2)}{ \sqrt{(e+g)^2-4(b+d)}}  \nonumber \\
&  \ \ \ \ \  \ \ \ \ \ + dp\frac{ \exp((e t)/2 + (g t)/2 + (t  \sqrt{(e+g)^2-4(b+d)})/2) (\sign(m - t) - 1)}{(e + g +  \sqrt{(e+g)^2-4(b+d)})  \sqrt{(e+g)^2-4(b+d)}}  \nonumber \\
&  \ \ \ \ \  \ \ \ \ \ - dp\frac{ \exp((e t)/2 + (g t)/2 + (t  \sqrt{(e+g)^2-4(b+d)})/2) (\sign(n - t) - 1)}{(e + g +  \sqrt{(e+g)^2-4(b+d)})  \sqrt{(e+g)^2-4(b+d)}}\biggr)  \nonumber \\
%%%%%%
& + \exp(-t (e/2 + g/2 -  \sqrt{(e+g)^2-4(b+d)}/2)) \times  \nonumber \\
& \times \frac{ (4 c - 4 a + 2 e v + 2 g v - u ((e+g)^2-4(b+d)) + e^2 u + g^2 u - 2 v  \sqrt{(e+g)^2-4(b+d)} + 2 e g u)}{2 (e + g -  \sqrt{(e+g)^2-4(b+d)})  \sqrt{(e+g)^2-4(b+d)}}  \nonumber  \\
%%%%%%%
& - \exp(-t (e/2 + g/2 +  \sqrt{(e+g)^2-4(b+d)}/2)) \times  \nonumber \\
& \times \frac{ (4 c - 4 a + 2 e v + 2 g v - u ((e+g)^2-4(b+d)) + e^2 u + g^2 u + 2 v  \sqrt{(e+g)^2-4(b+d)} + 2 e g u)}{2 (e + g +  \sqrt{(e+g)^2-4(b+d)})  \sqrt{(e+g)^2-4(b+d)}}  \nonumber 
\end{align}
that can be firstly arranged to
\begin{align}
y(t) & = \exp(-t \frac{e+ g-  \sqrt{(e+g)^2-4(b+d)}}{2})\times \\
& \ \ \ 
\times\biggl(2 a \frac{\exp(t\frac{e  + g  -   \sqrt{(e+g)^2-4(b+d)}}{2})}{(e + g -  \sqrt{(e+g)^2-4(b+d)})  
\sqrt{(e+g)^2-4(b+d)}}  \nonumber \\
& \ \ \ \ \  \ \ \ \ \ - 2 c \frac{\exp(t\frac{e + g  -  \sqrt{(e+g)^2-4(b+d)}}{2})}{(e + g -  \sqrt{(e+g)^2-4(b+d)})  
\sqrt{(e+g)^2-4(b+d)}}  \nonumber \\
&  \ \ \ \ \  \ \ \ \ \ - e p \heaviside(t - m) \frac{\exp(m\frac{e+g -\sqrt{(e+g)^2-4(b+d)}}{2})}{ \sqrt{(e+g)^2-4(b+d)}} 
\nonumber \\
& \phantom{\ \ \ \ \ \ \ H(t-m)}+ e p \heaviside(t - n) \frac{\exp(n\frac{ e+g-\sqrt{(e+g)^2-4(b+d)}}{2})}{ \sqrt{(e+g)^2-4(b+d)}}  \nonumber \\
&  \ \ \ \ \  \ \ \ \ \ + d p \frac{\exp(t\frac{e+g-\sqrt{(e+g)^2-4(b+d)}}{2}) (\sign(m - t) - 1)}{(e + g -  \sqrt{(e+g)^2-4(b+d)})  \sqrt{(e+g)^2-4(b+d)}}  \nonumber \\
&  \ \ \ \ \  \ \ \ \ \ - d p \frac{\exp(t\frac{e+g-\sqrt{(e+g)^2-4(b+d)}}{2}) (\sign(n - t) - 1)}{(e + g -  \sqrt{(e+g)^2-4(b+d)})  \sqrt{(e+g)^2-4(b+d)}}\biggr)  \nonumber \\
%%%%%%%
& \ \ \ - \exp(-t \frac{e+g+\sqrt{(e+g)^2-4(b+d)}}{2}) \times  \nonumber \\
& \ \ \
\times \biggl(2 a \frac{\exp(t\frac{e+g+\sqrt{(e+g)^2-4(b+d)}}{2})}{(e + g +  \sqrt{(e+g)^2-4(b+d)})  \sqrt{(e+g)^2-4(b+d)}}  \nonumber \\
&  \ \ \ \ \  \ \ \ \ \ - 2 c \frac{\exp(t\frac{e+g+\sqrt{(e+g)^2-4(b+d)}}{2})}{(e + g +  \sqrt{(e+g)^2-4(b+d)})  \sqrt{(e+g)^2-4(b+d)}}  \nonumber \\
&  \ \ \ \ \  \ \ \ \ \ - e p \heaviside(t - m) \frac{\exp(m\frac{e+g+\sqrt{(e+g)^2-4(b+d)}}{2})}{ \sqrt{(e+g)^2-4(b+d)}}
\nonumber \\
& \phantom{\ \ \ \ \ \ \ H(t-m)}+ e p \heaviside(t - n) \frac{\exp(n\frac{e+g+\sqrt{(e+g)^2-4(b+d)}}{2})}{ \sqrt{(e+g)^2-4(b+d)}}  \nonumber \\
&  \ \ \ \ \  \ \ \ \ \ + dp \frac{\exp(t\frac{e+g+\sqrt{(e+g)^2-4(b+d)}}{2}) (\sign(m - t) - 1)}{(e + g +  \sqrt{(e+g)^2-4(b+d)})  \sqrt{(e+g)^2-4(b+d)}}  \nonumber  \\
&  \ \ \ \ \  \ \ \ \ \ - dp\frac{\exp(t\frac{e+g+\sqrt{(e+g)^2-4(b+d)}}{2}) (\sign(n - t) - 1)}{(e + g +  \sqrt{(e+g)^2-4(b+d)})  \sqrt{(e+g)^2-4(b+d)}}\biggr) \nonumber \\
%%%%%%
&  \ \ \ \ + \exp(-t \frac{e+g-\sqrt{(e+g)^2-4(b+d)}}{2}) \times  \nonumber \\
&  \ \ \ \ \ \ \ \ \times \frac{ 4 c - 4 a + 2 e v + 2 g v +4u(b+d) - 2 v  \sqrt{(e+g)^2-4(b+d)}}{2 (e + g -  \sqrt{(e+g)^2-4(b+d)})  \sqrt{(e+g)^2-4(b+d)}}  \nonumber \\
%%%%%%%
&  \ \ \ \ - \exp(-t \frac{e+g+\sqrt{(e+g)^2-4(b+d)}}{2}) \times  \nonumber \\
&  \ \ \ \ \ \ \ \ \times \frac{ 4 c - 4 a + 2 e v + 2 g v +4u(b+d)+ 2 v  \sqrt{(e+g)^2-4(b+d)}}{2 (e + g +  \sqrt{(e+g)^2-4(b+d)})  \sqrt{(e+g)^2-4(b+d)}}  \nonumber 
\end{align}
and then, introducing the exponential factors inside the big parentheses, we get
\begin{align}
y(t) & = \frac{2(a-c)}{(e + g -  \sqrt{(e+g)^2-4(b+d)})  
\sqrt{(e+g)^2-4(b+d)}} \\
& \ \ \ +dp \frac{\sign(m-t)-\sign(n-t)}{(e + g -  \sqrt{(e+g)^2-4(b+d)})  
\sqrt{(e+g)^2-4(b+d)}}  \nonumber \\
&  \ \ \ - e p \heaviside(t - m) \frac{\exp((m-t)\frac{e+g -\sqrt{(e+g)^2-4(b+d)}}{2})}{ \sqrt{(e+g)^2-4(b+d)}} 
\nonumber \\
& \phantom{\ \ \ \ \ \ \ H(t-m)}
+ e p \heaviside(t - n) \frac{\exp((n-t)\frac{ e+g-\sqrt{(e+g)^2-4(b+d)}}{2})}{ \sqrt{(e+g)^2-4(b+d)}}  \nonumber \\
%%%%%%%
&
\ \ \ - \frac{2(a-c)}{(e + g +  \sqrt{(e+g)^2-4(b+d)})  \sqrt{(e+g)^2-4(b+d)}}  \nonumber \\
& \ \ \ \ \ + e p \heaviside(t - m) \frac{\exp((m-t)\frac{e+g+\sqrt{(e+g)^2-4(b+d)}}{2})}{ \sqrt{(e+g)^2-4(b+d)}}
\nonumber \\
& \phantom{\ \ \ \ \ \ \ H(t-m)}
- e p \heaviside(t - n) \frac{\exp((n-t)\frac{e+g+\sqrt{(e+g)^2-4(b+d)}}{2})}{ \sqrt{(e+g)^2-4(b+d)}}  \nonumber \\
&  \ \ \ \ \ - dp \frac{\sign(m - t) - \sign(n-t)}{(e + g +  \sqrt{(e+g)^2-4(b+d)})  \sqrt{(e+g)^2-4(b+d)}}  \nonumber \\
%%%%%%
& \ \ \ \ \ + \exp(-t \frac{e+g-\sqrt{(e+g)^2-4(b+d)}}{2}) \times  \nonumber \\
& \ \ \ \ \ \ \ \ \times \frac{ 4 c - 4 a + 2 e v + 2 g v +4u(b+d) - 2 v  \sqrt{(e+g)^2-4(b+d)}}{2 (e + g -  \sqrt{(e+g)^2-4(b+d)})  \sqrt{(e+g)^2-4(b+d)}}  \nonumber \\
%%%%%%%
& \ \ \ \ - \exp(-t \frac{e+g+\sqrt{(e+g)^2-4(b+d)}}{2}) \times  \nonumber \\
&  \ \ \ \ \ \ \ \ \times \frac{ 4 c - 4 a + 2 e v + 2 g v +4u(b+d)+ 2 v  \sqrt{(e+g)^2-4(b+d)}}{2 (e + g +  \sqrt{(e+g)^2-4(b+d)})  \sqrt{(e+g)^2-4(b+d)}} \nonumber 
\end{align}

We replace the signum function as $\sign(t)=2H(t)-1$ and hence
\begin{equation*}\sign(m-t)-\sign(n-t)=2\bigl(H(m-t)-H(n-t)\bigr)\end{equation*} 
and get
\begin{align}
y(t) & = \frac{2(a-c)}{(e + g -  \sqrt{(e+g)^2-4(b+d)})  
\sqrt{(e+g)^2-4(b+d)}} \\
& \ \ \ +dp \frac{2\bigl(H(m-t)-H(n-t)\bigr)}{(e + g -  \sqrt{(e+g)^2-4(b+d)})  
\sqrt{(e+g)^2-4(b+d)}}  \nonumber  \\
&  \ \ \ + e p \heaviside(t - m) \frac{\exp((m-t)\frac{e+g -\sqrt{(e+g)^2-4(b+d)}}{2})}{ \sqrt{(e+g)^2-4(b+d)}} 
\nonumber \\
& \phantom{\ \ \ \ \ \ \ H(t-m)}
- e p \heaviside(t - n) \frac{\exp((n-t)\frac{ e+g-\sqrt{(e+g)^2-4(b+d)}}{2})}{ \sqrt{(e+g)^2-4(b+d)}}  \nonumber \\
%%%%%%%
&
\ \ \ - \frac{2(a-c)}{(e + g +  \sqrt{(e+g)^2-4(b+d)})  \sqrt{(e+g)^2-4(b+d)}}  \nonumber \\
& \ \ \ \ \ - e p \heaviside(t - m) \frac{\exp((m-t)\frac{e+g+\sqrt{(e+g)^2-4(b+d)}}{2})}{ \sqrt{(e+g)^2-4(b+d)}}
\nonumber \\
& \phantom{\ \ \ \ \ \ \ H(t-m)}
+ e p \heaviside(t - n) \frac{\exp((n-t)\frac{e+g+\sqrt{(e+g)^2-4(b+d)}}{2})}{ \sqrt{(e+g)^2-4(b+d)}}  \nonumber \\
&  \ \ \ \ \ - dp \frac{2\bigl(H(m-t)-H(n-t)\bigr)}{(e + g + \sqrt{(e+g)^2-4(b+d)}) \sqrt{(e+g)^2-4(b+d)}} \nonumber \\
%%%%%%
& \ \ \ \ \ + \exp(-t \frac{e+g-\sqrt{(e+g)^2-4(b+d)}}{2}) \times  \nonumber \\
& \ \ \ \ \ \ \ \ \times \frac{ 4 c - 4 a + 2 e v + 2 g v +4u(b+d) - 2 v  \sqrt{(e+g)^2-4(b+d)}}{2 (e + g -  \sqrt{(e+g)^2-4(b+d)})  \sqrt{(e+g)^2-4(b+d)}}  \nonumber \\
%%%%%%%
& \ \ \ \ - \exp(-t \frac{e+g+\sqrt{(e+g)^2-4(b+d)}}{2}) \times  \nonumber \\
&  \ \ \ \ \ \ \ \ \times \frac{ 4 c - 4 a + 2 e v + 2 g v +4u(b+d)+ 2 v  \sqrt{(e+g)^2-4(b+d)}}{2 (e + g +  \sqrt{(e+g)^2-4(b+d)})  \sqrt{(e+g)^2-4(b+d)}} \nonumber 
\end{align}
Then, we change the order in which we write the terms in order to point out in the first six rows the unperturbed 
solution and in the last four rows the new terms appearing in the perturbed solution
\begin{align} %%% constanta de echilbru
y(t) & = \frac{2(a-c)}{(e + g -  \sqrt{(e+g)^2-4(b+d)})  
\sqrt{(e+g)^2-4(b+d)}}  \label{e:yeteord}\\
& \ \ \ - \frac{2(a-c)}{(e + g +  \sqrt{(e+g)^2-4(b+d)})  \sqrt{(e+g)^2-4(b+d)}} \nonumber \\ 
%%%%%% solutia neperturbata
& \ \ \ \ \ + \exp(-t \frac{e+g-\sqrt{(e+g)^2-4(b+d)}}{2}) \times  \nonumber \\
& \ \ \ \ \ \ \ \ \times \frac{ 4 c - 4 a + 2 e v + 2 g v +4u(b+d) - 2 v  \sqrt{(e+g)^2-4(b+d)}}{2 (e + g -  \sqrt{(e+g)^2-4(b+d)})  \sqrt{(e+g)^2-4(b+d)}}  \nonumber \\
& \ \ \ \ - \exp(-t \frac{e+g+\sqrt{(e+g)^2-4(b+d)}}{2}) \times  \nonumber \\
&  \ \ \ \ \ \ \ \ \times \frac{ 4 c - 4 a + 2 e v + 2 g v +4u(b+d)+ 2 v  \sqrt{(e+g)^2-4(b+d)}}{2 (e + g +  \sqrt{(e+g)^2-4(b+d)})  \sqrt{(e+g)^2-4(b+d)}} \nonumber \\
%%%%%%% termeni noi din solutia perturbata
& \ \ \ +dp \frac{2\bigl(H(m-t)-H(n-t)\bigr)}{(e + g -  \sqrt{(e+g)^2-4(b+d)})  
\sqrt{(e+g)^2-4(b+d)}}  \nonumber \\
&  \ \ \ \ \ - dp \frac{2\bigl(H(m-t)-H(n-t)\bigr)}{(e + g +  \sqrt{(e+g)^2-4(b+d)})  \sqrt{(e+g)^2-4(b+d)}} 
 \nonumber \\
&  \ \ \ + e p \heaviside(t - m) \frac{\exp((m-t)\frac{e+g -\sqrt{(e+g)^2-4(b+d)}}{2})}{ \sqrt{(e+g)^2-4(b+d)}} 
\nonumber \\
& - e p \heaviside(t - n) \frac{\exp((n-t)\frac{ e+g-\sqrt{(e+g)^2-4(b+d)}}{2})}{ \sqrt{(e+g)^2-4(b+d)}}  \nonumber \\ 
& \ \ \ \ \ - e p \heaviside(t - m) \frac{\exp((m-t)\frac{e+g+\sqrt{(e+g)^2-4(b+d)}}{2})}{ \sqrt{(e+g)^2-4(b+d)}}
\nonumber \\
& \phantom{\ \ \ \ \ \ \ H(t-m)}
+ e p \heaviside(t - n) \frac{\exp((n-t)\frac{e+g+\sqrt{(e+g)^2-4(b+d)}}{2})}{ \sqrt{(e+g)^2-4(b+d)}}  \nonumber 
\end{align}
Then observe that
\begin{align}\label{e:acab}
\frac{a-c}{b+d} & = \frac{2(a-c)}{(e + g -  \sqrt{(e+g)^2-4(b+d)})  
\sqrt{(e+g)^2-4(b+d)}} \\
& \ \ \ - \frac{2(a-c)}{(e + g +  \sqrt{(e+g)^2-4(b+d)})  \sqrt{(e+g)^2-4(b+d)}},\nonumber
\end{align}
and that
\begin{align}\label{e:depedoi}
-\frac{dp}{b+d} & \bigl(H(t-m)-H(t-n)\biggr) \\
& = dp \frac{2\bigl(H(m-t)-H(n-t)\bigr)}{(e + g -  \sqrt{(e+g)^2-4(b+d)})  
\sqrt{(e+g)^2-4(b+d)}} \nonumber \\
& \ \ \ \ \ - dp \frac{2\bigl(H(m-t)-H(n-t)\bigr)}{(e+g +\sqrt{(e+g)^2-4(b+d)})  \sqrt{(e+g)^2-4(b+d)}}.\nonumber
\end{align}
In view of \eqref{e:acab}--\eqref{e:depedoi} it follows that
\begin{align} %%%% functia de cuplare
y(t) & = \frac{a-c}{b + d}  -\frac{dp}{b+d}\bigl(H(t-m)-H(t-n)\bigr) \label{e:yetedoi}\\
%%%%%% solutia neperturbata
& \ \ \ \ \ + \exp(-t \frac{e+g-\sqrt{(e+g)^2-4(b+d)}}{2}) \times  \nonumber \\
& \ \ \ \ \ \ \ \ \times \frac{ 4 c - 4 a + 2 e v + 2 g v +4u(b+d) - 2 v  \sqrt{(e+g)^2-4(b+d)}}{2 (e + g -  \sqrt{(e+g)^2-4(b+d)})  \sqrt{(e+g)^2-4(b+d)}}  \nonumber \\
& \ \ \ \ - \exp(-t \frac{e+g+\sqrt{(e+g)^2-4(b+d)}}{2}) \times  \nonumber \\
&  \ \ \ \ \ \ \ \ \times \frac{ 4 c - 4 a + 2 e v + 2 g v +4u(b+d)+ 2 v  \sqrt{(e+g)^2-4(b+d)}}{2 (e + g +  \sqrt{(e+g)^2-4(b+d)})  \sqrt{(e+g)^2-4(b+d)}} \nonumber \\
%%%%%%% termeni noi din solutia perturbata
&  \ \ \ \ \ \ + e p \heaviside(t - m) \frac{\exp((m-t)\frac{e+g -\sqrt{(e+g)^2-4(b+d)}}{2})}{ \sqrt{(e+g)^2-4(b+d)}}
\nonumber \\
& \phantom{\ \ \ \ \ \ H(t-m)}
- e p \heaviside(t - n) \frac{\exp((n-t)\frac{ e+g-\sqrt{(e+g)^2-4(b+d)}}{2})}{ \sqrt{(e+g)^2-4(b+d)}}  \nonumber \\ 
& \ \ \ \ \ - e p \heaviside(t - m) \frac{\exp((m-t)\frac{e+g+\sqrt{(e+g)^2-4(b+d)}}{2})}{ \sqrt{(e+g)^2-4(b+d)}}
\nonumber \\
& \phantom{\ \ \ \ \ \ H(t-m)}
+ e p \heaviside(t - n) \frac{\exp((n-t)\frac{e+g+\sqrt{(e+g)^2-4(b+d)}}{2})}{ \sqrt{(e+g)^2-4(b+d)}}. \nonumber
\end{align}
Note that, since all constants $b,d,e,g$ are positive we have
\begin{equation}
\lambda_1=\frac{e+g +\sqrt{(e+g)^2-4(b+d)}}{2},\quad \lambda_2=\frac{e+g -\sqrt{(e+g)^2-4(b+d)}}{2}\neq 0,
\end{equation}
but these may be either real or complex conjugate numbers. In case both these numbers are real it follows
that actually both of them are positive.

\subsection*{A.2: Negative Discriminant.}\label{ss:and}
 With notation as before, assume that 
\begin{equation*}\Delta=(e+g)^2-4(b+d)<0,\end{equation*} and
hence both roots $\lambda_1$ and $\lambda_2$ are complex conjugate nonreal numbers. 
In order to simplify the 
notation and to point out some important aspects, in this case we let the real numbers $\alpha$ and $\omega$
be defined by
\begin{equation}\label{e:alo}
\alpha=\frac{e+g}{2}>0,\quad \iac\omega=\frac{\sqrt{(e+g)^2-4(b+d)}}{2}.
\end{equation}
With this notation, we claim that
\begin{align}\label{e:doinep}
\exp(-\alpha t) & \biggl((u -\frac{a-c}{b+d})\cos \omega t+\frac{v  +
\alpha(u -\frac{a-c}{b+d})}{\omega}\sin \omega t\biggr) \\
&  = \exp(-t \frac{e+g-\sqrt{(e+g)^2-4(b+d)}}{2}) \times \nonumber  \\
 & \times \frac{ 4 c - 4 a + 2 e v  + 2 g v  +4u (b+d) 
- 2 v   \sqrt{(e+g)^2-4(b+d)}}{2 (e + g -  \sqrt{(e+g)^2-4(b+d)})  \sqrt{(e+g)^2-4(b+d)}} \nonumber \\
& - \exp(-t \frac{e+g+\sqrt{(e+g)^2-4(b+d)}}{2}) \times  \nonumber \\
&  \ \ \ \ \ \ \ \ \times \frac{ 4 c - 4 a + 2 e v  + 2 g v  +4u (b+d)
+ 2 v  \sqrt{(e+g)^2-4(b+d)}}{2 (e + g + \sqrt{(e+g)^2-4(b+d)})  \sqrt{(e+g)^2-4(b+d)}}, \nonumber
\end{align}
which is the homogeneous part of the unperturbed solution, compare to the formula (2.20) in 
(Gheondea-Eladi2016). 

Indeed, let us denote by $F$ the expression from the right side in \eqref{e:yetedoi}. Then, with notation as 
in \eqref{e:doinep}, we have
\begin{align*}
F & = \exp(-t\alpha+\iac \omega t) 
\frac{(c-a)+u (b+)+v  (\alpha-\iac\omega)}{(\alpha-\iac\omega) 2\iac\omega} \\
& \ \ \ \ \ -\exp(-t\alpha-\iac t\omega t)
\frac{(c-a)+u (b+)+v  (\alpha-\iac\omega)}{(\alpha+\iac\omega) 2\iac\omega} \\
& = \exp(-t\alpha) (\cos \omega t+\iac\sin \omega t) 
\frac{(c-a)+u (b+)+v  (\alpha-\iac\omega)}{(\alpha-\iac\omega) 2\iac\omega} \\
& \ \ \ \ \ -\exp(-t\alpha)(\cos \omega t-\iac \sin \omega t)
\frac{(c-a)+u (b+)+v  (\alpha-\iac\omega)}{(\alpha+\iac\omega) 2\iac\omega} \\
& = \frac{\exp(-t\alpha)}{2\iac\omega} \times \biggl[ (\cos \omega t+\iac\sin \omega t)
\biggl(\frac{(c-a)+u (b+d)}{\alpha-\iac\omega}+v \biggr) \\
& \phantom{=\frac{\exp(-t\alpha)}{2\iac\omega} \times- -}-(\cos \omega t-\iac\sin \omega t)
\biggl(\frac{(c-a)+u (b+d)}{\alpha+\iac\omega}+v \biggr)\biggr].
\end{align*}
Since,
\begin{equation*}
(\alpha+\iac\omega)(\alpha-\iac\omega)=\alpha^2+\omega^2=\frac{(e+g)^2}{4}+\frac{4(b+d)-(e+g)^2}{4}=b+d,
\end{equation*}
it follows that
\begin{align*}
F & = \frac{\exp(-t\alpha)}{2\iac\omega}\biggl[ (\cos \omega t+\iac \sin \omega t)
\biggl((u -\frac{a-c}{b+d})(\alpha+\iac\omega)+v \biggr) \\
& \ \ \ \ \ - (\cos \omega t-\iac\sin \omega t)
\biggl((u -\frac{a-c}{b+d})(\alpha-\iac\omega)+v \biggr)\biggr] \\
& = \frac{\exp(-t\alpha)}{2\iac\omega}\biggl[ (\cos \omega t+\iac \sin \omega t)
\biggl(\bigl((u -\frac{a-c}{b+d})\alpha+v \bigr)+\iac\omega(u -\frac{a-c}{b+d})\biggr) \\
& \ \ \ \ \ - (\cos \omega t-\iac\sin \omega t)
\biggl(\bigl((u -\frac{a-c}{b+d})\alpha+v \bigr)-\iac\omega(u -\frac{a-c}{b+d})\biggr)
\biggr].
\end{align*}
Observing that the expression from the rightmost side of the previous equality is of the type
\begin{equation*}
(x+\iac y)(p+\iac q)-(x-\iac y)(p-\iac q)= 2\iac (xq+yp),
\end{equation*}
it follows that
\begin{align*}
F & = \frac{\exp(-t\alpha)}{2\iac\omega} 2\iac \bigl[ \omega(u -\frac{a-c}{b+d})
\cos \omega t+ \bigl((u -\frac{a-c}{b+d})\alpha+v \bigr)\bigr] \\
& = \exp(-\alpha t) \biggl((u -\frac{a-c}{b+d})\cos \omega t+\frac{v  +
\alpha(u -\frac{a-c}{b+d})}{\omega}\sin \omega t\biggr),
\end{align*}
hence \eqref{e:doinep} is proven.

Further on, with notation as in \eqref{e:alo}, a straightforward calculation shows that, 
\begin{align}\label{e:doiper}
-\frac{ep}{\omega}\exp(-t\alpha) & \biggl[ \exp(m\alpha)H(t-m)\sin((m-t)\omega)-\exp(n\alpha)H(t-n)\sin((n-t)\omega)\biggr] \\
& = e p \heaviside(t - m) \frac{\exp((m-t)\frac{e+g -\sqrt{(e+g)^2-4(b+d)}}{2})}{ \sqrt{(e+g)^2-4(b+d)}}
\nonumber \\
& \phantom{\ \ \ \ \ \ H(t-m)}
- e p \heaviside(t - n) \frac{\exp((n-t)\frac{ e+g-\sqrt{(e+g)^2-4(b+d)}}{2})}{ \sqrt{(e+g)^2-4(b+d)}}  \nonumber \\ 
& \ \ \ \ \ - e p \heaviside(t - m) \frac{\exp((m-t)\frac{e+g+\sqrt{(e+g)^2-4(b+d)}}{2})}{ \sqrt{(e+g)^2-4(b+d)}}
\nonumber \\
& \phantom{\ \ \ \ \ \ H(t-m)}
+ e p \heaviside(t - n) \frac{\exp((n-t)\frac{e+g+\sqrt{(e+g)^2-4(b+d)}}{2})}{ \sqrt{(e+g)^2-4(b+d)}}. \nonumber
\end{align}
Consequently, from \eqref{e:thetadoi}, \eqref{e:doinep}, and \eqref{e:doiper}, we have
\begin{align}  %%%% functia de cuplare
y(t) & = \frac{a-c}{b+d}-\frac{dp}{b+d}\bigl(H(t-m)-H(t-n)\bigr) \label{e:thetadoi}\\
%%%%% solutia omogena neperturbata
& \ \ \ +\exp(-\alpha t)\biggl((u -\frac{a-c}{b+d})\cos \omega t+\frac{v  +
\alpha(u -\frac{a-c}{b+d})}{\omega}\sin \omega t\biggr) \nonumber \\
%%%%% termeni omogeni de perturbare
& \ \ \ -\frac{ep}{\omega}\exp(-t\alpha) 
\biggl[\exp(m\alpha) H(t-m)\sin((m-t)\omega)-\exp(n\alpha)H(t-n)\sin((n-t)\omega)\biggr]\nonumber\\
& = \frac{a-c}{b+d} +\exp(-\alpha t)\biggl((u -\frac{a-c}{b+d})\cos \omega t+\frac{v  +
\alpha(u -\frac{a-c}{b+d})}{\omega}\sin \omega t\biggr) \label{e:thetadoirev}\\
& \ \ \ -\frac{dp}{b+d}\bigl(H(t-m)-H(t-n)\bigr) \nonumber \\
& \ \ \ -\frac{ep}{\omega}\exp(-t\alpha) 
\biggl[\exp(m\alpha) H(t-m)\sin((m-t)\omega)-\exp(n\alpha)H(t-n)\sin((n-t)\omega)\biggr].\nonumber
\end{align}

In \eqref{e:thetadoi} we emphasised on the first row the equilibrium function and on the second and third row
the component which is faded away by the exponential with negative exponent. In \eqref{e:thetadoirev}
we emphasised on the first row the solution of the unperturbed IVP and on the second and
the third row the components added as a result of the perturbation.

Finally, using
\begin{equation*}
\sin((m-t)\omega)=\sin(m\omega)\cos(\omega t)-\cos(m\omega)\sin(\omega t),
\end{equation*}
and similarly for $n$ instead of $m$, we can represent the function $y$ as follows:
\begin{align}
y(t) & = \frac{a-c}{b+d}-\frac{dp}{b+d}\bigl(H(t-m)-H(t-n)\bigr) \label{e:thetadoiord}\\
& \ \ \ +\exp(-\alpha t) \times \nonumber \\
& \ \ \ \ \times\! \biggl[\! \biggl( (u -\frac{a-c}{b+d})\!-\!\frac{ep}{\omega}\bigl[ \exp(m\omega)\sin(m\omega) H(t-m)-\exp(n\omega)\sin(n\omega)H(t-n)\bigr]\!\biggr)\!\cos(\omega t) \nonumber \\
& \ \ \ \ \ \ \ \ +\biggl(\frac{v +\alpha(u -\frac{a-c}{b+d})}{\omega} \nonumber \\
& \ \ \ \ \ \ \ \ \ \ \ \ \ +\frac{ep}{\omega}\bigr[ \exp(m\omega)\cos(m\omega)H(t-m)-\exp(n\omega)\cos(n\omega)H(t-n)\big]\biggr)\sin(\omega t)\biggl].\nonumber
\end{align}
In this case, the above formula tells us that the solution has an evolution to the equilibrium function,
which is constantly equal to $\frac{a-c}{b+d}$, except on the interval $[m,n]$ where it has a gap, and that
there is a permanent oscillation of frequency $\omega\neq 0$ plus some perturbations on the interval $[m,n]$
and a little beyond.
%%%%%

\subsection*{A.3: The Solution of the 2nd Order ODE with Forcing Factor.}\label{ss:sod}
In order to simplify the notation and to emphasise certain factors that have interpretation in the framework
of the evolution of confidence, with respect to the notation as in the previous subsection, we denote
\begin{equation}\label{e:nota}
\alpha=\frac{e+g}{2},\quad \beta=b+d,\quad \gamma=a-c.
\end{equation}
Then, the IVP \eqref{e:2ode} becomes
\begin{align}\label{e:2odea}
y^{\prime\prime}(t) +2\alpha y^\prime(t)+ \beta y(t) & =\gamma-dp(H(t-m)-H(t-n))+ep(\delta(t-m)-\delta(t-n)), \\
y(0)  = u, \quad & \quad
y^\prime(0) = v. \nonumber
\end{align}

\begin{proposition}\label{p:2ode}
The IVP \eqref{e:2odea} has unique solution defined as follows: if $\alpha^2-\beta\neq 0$ then
\begin{align}\label{e:2odesol}
y(t) & = \frac{\gamma}{\beta} - \frac{dp}{\beta}\bigl(H(t-m)-H(t-n)\bigr)\\
& + \frac{\emath^{-\alpha t}}{2\sqrt{\alpha^2-\beta}}
\biggl[ \bigl((u-\frac{\gamma}{\beta})(\alpha+\sqrt{\alpha^2-\beta})+v\bigr)  
\emath^{t\sqrt{\alpha^2-\beta}} \nonumber 
+ \bigl((u-\frac{\gamma}{\beta})(-\alpha+\sqrt{\alpha^2-\beta})-v\bigr)
\emath^{-t\sqrt{\alpha^2-\beta}}\biggr]\nonumber \\
& + \frac{ep\emath^{-\alpha t}}{2\sqrt{\alpha^2\!-\!\beta}}\biggl[H(t\!-\!m)\emath^{m\alpha}
\bigl(\emath^{-(m-t)\sqrt{\alpha^2\!-\!\beta}}\!-\!\emath^{(m-t)\sqrt{\alpha^2\!-\!\beta}}\bigr)
 - H(t\!-\!n)\emath^{n\alpha}
\bigl(\emath^{-(n-t)\sqrt{\alpha^2\!-\!\beta}}\!-\!\emath^{(n-t)\sqrt{\alpha^2\!-\!\beta}}\bigr)\biggr]\nonumber
\end{align}
and, if $\alpha^2-\beta= 0$ then
\begin{align}\label{e:2odesoldisnul} y(t) & = 
\frac{\gamma}{\alpha^2}-\frac{dp}{\alpha^2}\bigl[H(t-m)-H(t-n)\bigr] \\
& \phantom{\frac{\gamma}{\alpha^2}+}
+\emath^{-\alpha t}\biggl[ 
(u-\frac{\gamma}{\alpha^2})+\bigl(v-(\frac{\gamma}{\alpha^2}-u)\alpha\bigr)t %\nonumber \\
% & \phantom{\frac{\gamma}{\alpha^2}+} 
-ep(m-t)H(m-t)\emath^{m\alpha}+ep(n-t)H(n-t)\emath^{n\alpha}
\biggl] .\nonumber
\end{align}
In addition, if $\alpha^2-\beta<0$ then, letting $\omega=\sqrt{\beta-\alpha^2}$, the solution $y$ defined
as in \eqref{e:2odesol} has the following representation
\begin{align}
y(t) & = \frac{\gamma}{\beta}-\frac{dp}{\beta}\bigl(H(t-m)-H(t-n)\bigr) \label{e:2odesolcom}\\
& \ \ \ +\emath^{-\alpha t} \times \nonumber \\
& \ \ \ \ \times\! \biggl[\! \biggl( (u -\frac{\gamma}{\beta})\!-\!\frac{ep}{\omega}\bigl[ \emath^{m\omega}
\sin(m\omega) H(t-m)-\emath^{n\omega}\sin(n\omega)H(t-n)\bigr]\!\biggr)\!\cos(\omega t) \nonumber \\
& \ \ \ \ \ \ \ \ +\biggl(\frac{v +\alpha(u -\frac{\gamma}{\beta})}{\omega}  +\frac{ep}{\omega}\bigl[ \emath^{m\omega}\cos(m\omega)H(t-m)-\emath^{n\omega}
\cos(n\omega)H(t-n)\bigr]\biggr)\sin(\omega t)\biggr].\nonumber
\end{align}
\end{proposition}

\begin{proof} Assuming that $\alpha^2-\beta\neq 0$, the formula \eqref{e:2odesol} is a consequence of
the formula \eqref{e:yetedoi}, with some additional calculations that we leave to the reader.

Assuming that $\alpha^2-\beta<0$, the formula \eqref{e:2odesolcom} is a consequence of the formula
\eqref{e:thetadoiord}.

It remains to prove that, in case $\alpha^2-\beta=0$ the formula \eqref{e:2odesoldisnul} provides the solution
to the IVP \eqref{e:nota}. To this end, 
since the 2nd order linear ODE in \eqref{e:nota} has constant coefficients and is nonhomogeneous, it has a
unique global solution which is obtained as the sum of a particular solution $y_{\mathrm{p}}(t)$
and the general solution $y_\mathrm{g}$ of the associated homogeneous ODE
\begin{equation}\label{e:2odehom}
y^{\prime\prime}+2\alpha y^\prime+\beta=0.
\end{equation}
In the particular case $\alpha^2-\beta=0$
the characteristic polynomial of the homogeneous 2nd order ODE associated to \eqref{e:2ode} 
has only one root $\lambda=-\alpha$ of multiplicity $2$ and hence, the general solution of the ODE 
\eqref{e:2odehom} is 
\begin{equation*}
y_\mathrm{g}(t)=C_1 \exp(-\alpha t)+C_2t\exp(-\alpha t),
\end{equation*}
and we have to determine the constants $C_1$ and $C_2$ such that initial conditions $y(0)=u$ and 
$y^\prime(0)=v$ are fulfilled.

On the other hand, in order to obtain a particular solution $y_\mathrm{p}$ 
we use the information from \eqref{e:2odesol} and claim that
\begin{equation}\label{e:2solpar}
y_\mathrm{p}(t)=
\frac{\gamma}{\alpha^2}-\frac{dp}{\alpha^2}\bigl[H(t-m)-H(t-n)\bigr] 
-ep\emath^{-\alpha t}\bigl[
(m-t)H(m-t)\emath^{m\alpha}-(n-t)H(n-t)\emath^{n\alpha}
\bigl] .\nonumber
\end{equation}
This can be verified in at least two different fashions: first, by performing the calculations for each of the
three cases, $t\in (-\infty,m)$, $t\in (m,n)$, and $t\in(n,+\infty)$ and showing that it satisfies the 2nd order ODE
in \eqref{e:nota} and, second, by keeping all parameters and $t$ fixed, except $\alpha^2-\beta\ra 0$ and 
showing
that the nonhomogeneous part from the right hand side of \eqref{e:2odesol} converges to $y_\mathrm{p}(t)$.
Then, we determine the coefficients $C_1$ and $C_2$ such that $y=y_\mathrm{g}+y_\mathrm{p}$ satisfies
the initial conditions $y(0)=u$ and $y^\prime(0)=v$ and get the formula \eqref{e:2odesoldisnul}.
\end{proof}

\end{document}